\documentclass[12pt]{article}

\usepackage[numbers,sort]{natbib}

\usepackage{fullpage}

\usepackage[pdftex]{graphicx}
\usepackage{amsfonts}
\usepackage{amscd}
\usepackage{indentfirst}
\usepackage{amssymb,amsthm, amsgen, amstext, amsbsy, amsopn,yfonts}
\usepackage{amsmath}
\usepackage{mathrsfs}
\usepackage{ulem} %

\newtheorem{thm}{Theorem}[section]
\newtheorem{lemma}[thm]{Lemma}
\newtheorem{prop}[thm]{Proposition}
\newtheorem{cor}[thm]{Corollary}
\newtheorem{defin}[thm]{Definition}
\newtheorem{rem}[thm]{Remark}

\newtheorem{question}[thm]{Question}

\newtheorem{claim}{Claim}

\newcommand{\N}{{\mathbb{N}}}

\newcommand{\R}{{\mathbb{R}}}

\newcommand{\Z}{{\mathbb{Z}}}

\newcommand{\s}{{\text{stab}}}

\newcommand{\supp}{{\rm supp\ }}

\newcommand{\bal}{\begin{aligned}}
\newcommand{\enbal}{\end{aligned}}
\newcommand{\be}{\begin{equation}}
\newcommand{\ee}{\end{equation}}

\usepackage{latexsym,amsmath,amsthm,amssymb,amscd,amsfonts}
\usepackage[usenames]{color}
\usepackage{tikz,tikz-3dplot,pgfplots}  

\usetikzlibrary{decorations.pathreplacing,calligraphy}
\usetikzlibrary{decorations.markings}
\usetikzlibrary{decorations.pathmorphing}

\usetikzlibrary{3d} 
\usetikzlibrary{arrows,calc}
\usepackage{dsfont} 
\usepackage{subcaption}
\usepackage{array}
\usepackage{stackrel} 

\usepackage{hyperref} 
\usepackage{float} 
\usepackage{wrapfig}
\usepackage{mathtools} 

\usetikzlibrary{hobby}
\usepgfplotslibrary{fillbetween}
\usetikzlibrary{patterns}

\DeclareMathOperator{\const}{const}
\DeclareMathOperator{\im}{Im}
\DeclareMathOperator{\Area}{Area}
\DeclareMathOperator{\Ham}{Ham}
\DeclareMathOperator{\Hameo}{Hameo}

\DeclareMathOperator{\node}{node}
\DeclareMathOperator{\Aut}{Aut}
\DeclareMathOperator{\osc}{osc}

\DeclareMathOperator{\FQM}{FQM}

\newcommand{\lrang}[1]{
    \left\langle #1 \right\rangle
}

\newcommand{\pd}[2]{
    \frac{\partial #1}{\partial #2}
}

\renewcommand{\S}{{S}}

\newcommand{\id}{{\bf{1}}}

\newcommand{\SSplit}[1]{
    \begin{equation*}
        \begin{split}
            #1
        \end{split}
    \end{equation*}
}
\newcommand{\stackL}[1]{
    \begin{array}{l}
        #1
    \end{array}
}

\DeclarePairedDelimiter{\absx}{\lvert}{\rvert}
\DeclarePairedDelimiter{\normx}{\lVert}{\rVert}
\newcommand{\norm}[1]{\normx*{#1}}
\newcommand{\abs}[1]{\absx*{#1}}

\renewcommand{\leq}{\leqslant}
\renewcommand{\geq}{\geqslant}

\definecolor{BDarkBlue}{HTML}{0653A1}
\definecolor{gray75}{gray}{0.75}

\hypersetup{
    colorlinks,
    citecolor=black,
    filecolor=black,
    linkcolor=black,
    urlcolor=blue,
}

\newcommand{\setlabel}[1]{\edef\@currentlabel{#1}\label}
\newcommand{\emphref}[1]{\emph{{\ref{#1}}}}

\newcommand{\om}{\omega}
\newcommand{\eps}{\epsilon}

\def\mrm#1{{\mathrm{#1}}}

\def\cl#1{{\mathcal{#1}}}

\renewcommand{\r}{\rho}

\newcommand{\til}[1]{\widetilde{#1}}
\newcommand{\wh}[1]{\widehat{#1}}
\newcommand{\ol}[1]{\overline{#1}}
\newcommand{\ar}[1]{\vec{#1}}

\newcommand{\sm}[1]{C^\infty(#1)}

\begin{document}

\title{A dichotomy for the Hofer growth of area preserving maps on the sphere via symmetrization}
\author{Lev Buhovsky, Ben Feuerstein, Leonid Polterovich, and Egor Shelukhin}
\date{Mar 12, 2025}


\maketitle

\begin{abstract}
We prove that autonomous Hamiltonian flows on the two-sphere $S^2$ exhibit the following dichotomy: the Hofer norm either grows linearly in time or is bounded in time by a universal constant $C = 19 \mrm{Area}(S^2).$ Our approach involves a new technique, Hamiltonian symmetrization. Essentially, we prove that every autonomous Hamiltonian diffeomorphism is conjugate to an element $C$-close in the Hofer metric to one generated by a function of the height.
\end{abstract}

\tableofcontents

\section{Introduction and main results} 

Let $(M,\omega)$ be a closed connected symplectic manifold, i.e., an even dimensional manifold $M$ equipped with a closed non-degenerate 2-form $\omega$.
Every smooth function $F : M \times \R \to \R$ defines a time-dependent vector field $X^t_F$ which
solves the linear equation $i_{X^t_F} \om = -dF_t$ for $F_t: M \to \R$ given by $F_t(x) = F(x,t).$ This field defines a flow of diffeomorphisms $f_t : M \to M$. We say that {\it the Hamiltonian function}  $F$ generates the {\it Hamiltonian diffeomorphism} $\phi=f_1$. Hamiltonian diffeomorphisms form a group denoted by $\Ham (M,\omega)$. The interest in Hamiltonian diffeomorphisms is due to the fact that they model motions of classical mechanics on the phase space $M$, and that they preserve the symplectic form $\omega$. In fact, if $H^1(M,\R)=0$ (e.g. for $M=S^2$), $\Ham(M,\omega)$ is simply the identity component of the symplectomorphism group of $(M,\omega)$.

A remarkable discovery of modern symplectic topology
is that $\Ham(M,\omega)$ carries a bi-invariant Finsler metric with non-degenerate distance function $d(\phi,\psi)$. It is called {\it the Hofer metric} and is defined by $d({\bf 1} ,\phi) = \inf \int_0^1 ||F_t|| dt$, where the infimum is taken  over all Hamiltonian functions $F,$ with zero mean $\int F_t \om^n$ for all $t,$ generating $\phi.$ Here $||H|| = \max_{x \in M} |H(x)|$ is the uniform norm of a function on $M$. This metric was introduced by Hofer \cite{H}, who proved its non-degeneracy for the linear symplectic space (see also \cite{V}). This was extended in \cite{P-ETDS} to symplectic manifolds whose symplectic form represents a rational cohomology class, and finally
the non-degeneracy of $d$ was established in \cite{LM} for general symplectic manifolds.
We refer to \cite{P-book} for more details on Hofer's metric.

Given a one-parameter subgroup $\{f_t\}$ of $\Ham(M,\omega)$ generated by a {\it time-independent}
Hamiltonian $F : M \to \R$, we note that the function $t \mapsto d({\bf 1} , f_t)$ is subadditive and
hence the following limit exists:
\begin{equation}
\r(F):= \lim_{t \to +\infty} \frac{d({\bf 1} , f_t)}{t}\;.
\end{equation}
If $\r(F) >0$, we say that $\{f_t\}$ has {\it linear growth}. Linear growth takes place for $C^\infty$-generic functions $F$, at least on symplectically aspherical symplectic manifolds $(M,\om),$ see \cite[Chapter 6.3]{PR-book}. We will also write $\rho(\phi^1_F)=\rho(F)$ to emphasize the dependence on the Hamiltonian diffeomorphism only.

Interestingly enough, in all examples known to us, one-parameter subgroups enjoy the following dichotomy: either $\r(F)>0$ so the growth is linear, or $ d({\bf 1} , f_t)$ is a bounded function of $t$ and $\r(F)=0$. The latter option takes place, for instance, if
$\psi^*F=-F$ for some $\psi \in \Ham(M,\omega)$; such Hamiltonian flows are of interest in
dynamical systems and have a special name - reversible systems. This has led to the following question, see \cite[Question 3.3]{P-groups}.

\begin{question} \label{quest-1} Do there exist one-parameter subgroups of intermediate growth?
\end{question}

Our main result proves the following {\it enhanced dichotomy} for $M=S^2$,
and hence provides a negative answer to Question \ref{quest-1}.

\begin{thm}[Enhanced dichotomy]\label{thm-1} Consider the two-dimensional sphere $S^2$
equipped with a smooth area form. Then for every $F \in C^\infty(M)$
with $\r(F)=0$ and all $t$ in $\R$ we have
\begin{equation}\label{eq-main-1}
d({\bf 1} , f_t) \leq 19 \Area(S^2)\;.
\end{equation}
\end{thm}

A few remarks are in order. The fact that the upper bound in \eqref{eq-main-1}
does not depend on the choice of $F$ is unexpected to us. The constant $19$ is clearly not
optimal - it would be interesting to explore this issue. We expect that our method will yield
the (not necessarily enhanced) dichotomy for all surfaces, either closed or open. In the open case one deals with compactly supported Hamiltonians only. Let us mention that in the case of {\it open surfaces of infinite area} the dichotomy was established in \cite{PS} by different tools.

In order to describe our approach to the enhanced dichotomy, let's look at $\Ham(S^2)$ as
at an infinite-dimensional Lie group. Let us mention that this viewpoint highlights
the following {\it genuinely infinite-dimensional} feature: this group is simple by a result due to Banyaga, it carries a bi-invariant Finsler metric, and it is by no means compact.  For finite-dimensional Lie groups, this is impossible due to (a minor modification of) a result by Milnor, see \cite[Section 1.3.3]{PR-book} for a detailed discussion. Let $z: S^2 \to \R$
be the height function with $\max z = 1/2,\; \min z = -1/2$. Denote by $\mathcal{T} \subset \Ham(S^2)$ the group of diffeomorphisms generated by Hamiltonians of the form $u(z)$, and by $\mathcal{T}_{ev}$ the subgroup of $\mathcal{T}$ generated by {\it even} functions, i.e., $u(z)=u(-z)$. It is instructive (see \cite[Proposition (2.2)]{BFR}) to think of $\mathcal{T}$ as the maximal torus of $\Ham(S^2)$, and of $\mathcal{T}_{ev}$ as its subgroup of invariant elements with respect to the action of the Weyl group (as defined in \cite[p. 158]{BD}). (In our particular situation, the Weyl group is isomorphic to $\Z/2\Z.$)

Next, denote by $(\widehat{\Ham}(S^2), \widehat{d})$ the completion of $\Ham(S^2)$ with respect to Hofer's metric $d$, and write $\widehat{\mathcal{T}_{ev}}$ for the closure of $\mathcal{T}_{ev}$ in this completion. We shall prove the following result.

\begin{thm}\label{thm-2} $\;$
\begin{itemize}
\item[{(i)}] For every non-trivial one-parameter subgroup $\{f_t\}$ in $\Ham(S^2)$ there exists unique one-parameter subgroup $\{g_t\}$ in $\widehat{\mathcal{T}_{ev}}$
and a family of elements $\{\psi_t\}$ in ${\Ham}(S^2)$ such that
\begin{equation}\label{eq-thm2}
\widehat{d}( \psi_t f_t \psi_t^{-1}, g_t) \leq 19\cdot \Area(S^2).
\end{equation} 
\item[{(ii)}] The subgroup $\{g_t\}$ is either trivial (i.e., $g_t = {\bf 1}$ for all $t$)
or it has linear growth:
$$\lim_{t \to +\infty} \frac{\widehat{d}({\bf 1} , g_t)}{t} > 0\;.$$
\end{itemize}
\end{thm}


Recall that in a compact simple finite-dimensional Lie group every element is conjugate
to an element from a given maximal torus, see \cite[p. 159]{BD}. Theorem \ref{thm-2}(i)
can be considered to be a substitute of this result for autonomous diffeomorphisms in the group $\Ham(S^2)$: the conjugation
exists up to a bounded error and passing to the completion. Its proof is based on a
quite substantial generalization of the Sikorav's trick, see \cite{Sik, Buh}.
The proof of Theorem \ref{thm-2}(ii) involves the technique of Lagrangian estimators coming
from Lagrangian Floer theory in symmetric products \cite{PSh, CHMSS}, and in particular
on quasi-morphisms on $\widehat{\Ham}(S^2)$ constructed in \cite{CHMSS}.  
The enhanced dichotomy (Theorem \ref{thm-1}) readily follows from Theorem \ref{thm-2}.

\begin{rem}
The metric group $(\widehat{\Ham}(S^2), \widehat{d})$ plays an essential role in our approach to Question \ref{quest-1} and in our arguments. Therefore we see an importance in studying metric/group theoretic properties of $ \widehat{\Ham}(S^2) $. One such property is discussed in Remark \ref{remark:hat-Ham-non-simple} below.
\end{rem} 

\medskip\noindent
{\sc Organization of the paper:} In Section \ref{sec: symm} we define the key notion of the paper, Hamiltonian symmetrization, prove a few of its basic properties, and discuss three points of view on it: symmetrization of suitably adapted, so-called ``flattened quasi-Morse", Hamiltonian functions (Section \ref{subsec: symm fqm}), symmetrization by means of quasi-morphisms which applies to arbitrary continuous function (Section \ref{subsec: symm qm}), and a combinatorial description in terms of the Reeb graph (Section \ref{subsec: combin}). In Section \ref{sec: proof thm} we prove Theorem \ref{thm-2} by means of establishing a number of more refined properties of the symmetrization which we summarize in Section \ref{subsec: state thm}: uniform continuity as $z \to \pm 1/2$ for sufficiently smooth Hamiltonians (Section \ref{subsec: uniform}), a H\"older style inequality for comparing symmetrizations of nearby functions (Section \ref{subsec: Holder}), an estimate on the effect of flattening a Morse function on the symmetrization (Section \ref{subsec: flatten}), reduction of the theorem to the case of flattened quasi-Morse functions (Section \ref{subsec: thm mod fqm}), and finally its proof for these functions (Section \ref{subsec: thm fqm}).


\section*{Acknowledgments} LB was partially supported by the Israel Science Foundation grant 1918/23. BF and LP were partially supported by the ISF-NSFC grant 3231/23. ES was partially supported by the NSERC, Fondation Courtois, and Alfred P. Sloan Foundation.

\section{The symmetrization map}\label{sec: symm}

\noindent Our ultimate goal in this section is to define a map \[\Sigma:C_0^0(\S^2)\rightarrow C_{\text{even},0}^0\left(I\right)\] for the interval $I =\left(-\frac{1}{2},\frac{1}{2}\right),$ which we call the \textbf{ symmetrization map,} which enjoys various useful properties. Here $C_0^0(\S^2)$ is the set of mean-zero $C^0(\S^2)$ functions, and $C_{\text{even},0}^0\left(I\right)$ is the set of continuous mean-zero even functions on $I.$

For instance, we prove the following \emph{enhanced dichotomy} property, which shows how the symmetrization of an autonomous Hamiltonian $H \in C^2_0(S^2)$ is related to the Hofer growth of the one-parametric group $\{ \phi^t_H \}_{t \in \R}$: \[ \r(H) =0 \iff \Sigma(H)\equiv 0 \iff \forall t \quad \widehat{d}(\id,\varphi_H^t)< 19 \Area(\S^2).\]

Unfortunately, $\Sigma$ has the disadvantage of having a somewhat involved definition, and the rest of this section is dedicated to it. In order to be as concise as possible, the definition of $\Sigma$ is split into two parts, and followed by a quick discussion. In the first part, we define ``(flat) quasi-Morse" functions to be functions of Morse functions of the sphere (constant near critical levels), and give a rather intuitive definition of $\Sigma$ on these functions. In the second part we use a density argument to define the symmetrization for $C_0^0(\S^2)$ functions. This requires rather heavy machinery, notably quasi-morphisms and quasi-states from \cite{CHMSS} (in the spirit of \cite{PSh}). A running motif in these parts is that $\Sigma$ can be defined for functions of trees, and extended to functions on $\S^2$ using the theory of Reeb Graphs (see \cite{izosimov2016coadjoint}). This idea is expanded upon in the last subsection.

\subsection{Symmetrization of flat quasi-Morse functions}\label{subsec: symm fqm}
We will start with some auxiliary definitions and notations.
\begin{defin}
    A function $H \in C^\infty(\S^{2})$ is called {\bf quasi-Morse} (QM) if it is the pullback of a function on the Reeb graph of a Morse function: that is, if there exists a Morse function $f: M \rightarrow \R$ and a function \[ \til{H} : \Gamma_f \rightarrow \R \] such that \[H(q) = \til{H}(\pi(q))\] for all $q \in \S^2,$ where $\Gamma_f$ is the Reeb graph of $f$ and $\pi: M \rightarrow \Gamma_f$ is the natural projection. Furthermore, it is called {\bf flat quasi-Morse} (FQM) if for every vertex $v$ of $\Gamma_f$ of degree at least $2$ there is a neighborhood of $\pi^{-1}(v)$ on which $H$ is constant.
\end{defin}

\noindent
We will use the following terminology:
\begin{itemize}
\item We will say that the FQM function $H$ is subordinate to the Morse function $f.$ Note that such a function $f$ is not unique. 
    \item 
    Denote by $\FQM_0(M)\subset C^{\infty}(M)$ the set of mean-zero FQM functions on $M$ subordinate to a Morse function with distinct critical values.
    \item
    For $H$ Morse, and $\varepsilon>0$, let $r: \R \rightarrow \R$ be a smooth function which is constant near the critical values of $H$ corresponding to saddles which satisfies $\norm{r-\id_{\R}}_{C^0}<\varepsilon.$ We call the FQM function $\widetilde{H} = r\circ H$ the {\bf $\varepsilon$-flattening of $H$}.\\
\end{itemize} 
\begin{thm}\label{thm: fqm0 symm}
        There exists a map $\widetilde{\Sigma}: \FQM_0 (\S^2)\rightarrow C_{\text{even},0}^{\infty}\left(I \right)$, where $C_{\text{even},0}^{\infty}\left(I \right)$ are smooth mean-zero even functions of $I=\left(-\frac{1}{2},\frac{1}{2}\right)$ with the following properties:
        \begin{enumerate}
            \item \emph{(Homogeneity and symmetry)} \setlabel{Homogeneity and symmetry}{prop:AlgebraicPropertiesFQM}For all $H\in \FQM_0(\S^2)$ and $t \in \R$ we have
            \[
                \widetilde{\Sigma}(tH) = t\widetilde{\Sigma}(H),
            \]
            and if $H(z)$ is a function of the height, then
            \[
                \widetilde{\Sigma}(H) =  \frac{H(z) + H(-z)}{2}.
            \]
            \item \emph{(FQM proximality)} \setlabel{FQM proximality}{prop:HoferControlFQM} For every $H\in \FQM_0(\S^2)$ there exists a constant $C(H) \geq 0$ such that for all $t\in \R,$
            \[d\left(\varphi_H^t, \varphi_{\widetilde{\Sigma}(H)}^t\right) \leq C(H).\] In particular $\rho(H) = \rho(\til{\Sigma}(H)).$
        \end{enumerate}
\end{thm}

\noindent In the sequel we will identify functions $H$ on $I$ and functions $H \circ z$ on $S^2$ depending on the height only. Before proceeding with the proof of Theorem \ref{thm: fqm0 symm}, we note that \emphref{prop:HoferControlFQM} along with the fact that $\widetilde{\Sigma}(H) \not\equiv 0 $ implies $\r\left(\widetilde{\Sigma}(H)\right)>0$ (we will show this implication a bit later) proves the (not yet enhanced) dichotomy theorem for $\FQM_0(\S^2)$ functions.

{\it Construction of $\til{\Sigma}(H).$} Before proving Theorem \ref{thm: fqm0 symm}, we introduce the following preliminary construction. Let $H\in \FQM_0(\S^2)$ be subordinate to a Morse function $f.$ We will construct $\widetilde{\Sigma}(H)$ directly as follows. If $f$ has no saddles, then it is already symplectomorphic to a function of $z$ (this will, for instance, be evident from the arguments below) and hence so is $H.$ Denote by $F(z)$ the latter function. It remains to set $\widetilde{\Sigma}(H) = (F(z)+F(1/2-z))/2.$ Assume now that $f$ has a saddle point. 

Recall that $\Gamma_f$ denotes the Reeb graph of $f$ and  $\pi: \S^2 \rightarrow \Gamma_f$ is the natural projection. Let $\{A_i\}$ be the edges of $\Gamma_f$ where vertices of degree $3$ are excluded and vertices of degree $1$ are included (for example if an edge connects two vertices of degree $3,$ it is taken to be open). Set $S_i = \pi^{-1}(A_i) \subset \S^2.$ Then the $\{S_i\}$ are disks/annuli which are invariant under the flow of $H,$ and whose closures form a cover of $\S^2$ with non-empty intersections consisting of figure-eights lying on critical levels of $f.$ For every $i,$ let $H_i$ be the extension by constants of $H|_{S_i}$ and $\varphi_i^t:=\varphi_{H_i}^t \in {\Ham}(S^2).$ Note that as $H$ is an FQM, all the $H_i$ are smooth functions. Indeed, there exist disks or annuli $S'_i$ whose closures are contained in $S_i,$ such that $H$ is constant on the connected components of $S^2 \setminus \cup_i S'_i$ so the extension by constants of $H|_{S'_i}$ is smooth and coincides with that of $H|_{S_i}.$ Moreover, $\varphi_H^t = \prod_i \varphi^t_i$. We want to map $S'_i$ to a standard cap/annulus with a Hamiltonian diffeomorphism $\theta_i$ such that $H_i \circ  \theta_i^{-1}$ is a function of the height.  To construct $\theta_i$ we use \cite[Lemma 2.5]{Schlenk_2003}, which we recall here along with relevant definitions for the reader's convenience.

\begin{defin}
    A family $\mathcal{L}$ of loops in a simply connected domain $U \subset \R^2$ is called admissible if there is a diffeomorphism $\beta: D(U) \setminus \{0\} \rightarrow U \setminus \{p\}$, for some $p\in U$, where $D(U)$ is an open disk of area $\Area(U)$ around the origin, such that:
    \begin{enumerate}
        \item Concentric circles are mapped to elements of $\mathcal{L}$.
        \item In a neighborhood of the origin $\beta$ is a translation.
    \end{enumerate}
\end{defin}
\begin{lemma}\label{lma: felix}
Let $U$ and $V$ be bounded and simply connected domains in $\R^2$ of equal area and let $\mathcal{L}_U$ and $\mathcal{L}_U$ be admissible families of loops in $U$ and $V$, respectively. Then there is a symplectomorphism between $U$ and $V$ mapping loops to loops.
\end{lemma}

Our admissible family of loops will be the regular level sets of $H_i$ extended to a family of smooth loops on a neighborhood of the union of these regular level sets. We will map this family of loops to concentric circles $\{z=c\}$ in $W_i =\{ c_i \leq z \leq d_i\}\subset \S^2.$ Here $c_i<d_i$ are given as follows. If $S'_i$ is an annulus, consider the areas $c'_i, d'_i$ of the two connected components of the complement $S^2 \setminus S'_i$ in arbitrary order. Then $c_i = c'_i, d_i = 1- d'_i.$ Note that $c_i < d_i$ as $c'_i+d'_i < 1.$ If $S'_i$ is a disk, we proceed analogously. By Lemma \ref{lma: felix}, we construct an area-preserving diffeomorphism $\theta'_i: S'_i \to W_i$, which by \cite[Lemmas 2 and 4]{EPP} extends to a Hamiltonian diffeomorphism $\theta_i: \S^2 \to \S^2.$ 

For a continuous function $F \in C^0(S^2)$ set $\langle F \rangle = \frac{1}{\Area(S^2)} \int_{S^2} F \om$ for its mean. Now, we are ready to define $\widetilde{\Sigma}(H)$. Consider $R \in \Ham(\S^2)$ given via the embedding $\S^{2} \subset \R^3$ as the standard sphere of radius $1/2,$ by $(x,y,z) \mapsto (-x,y,-z)$.
\begin{defin}\label{def: fqm0 symm}
    Let $H \in \FQM_0(\S^2)$ and let $\{\theta_i\}\subset \Ham(\S^2)$ be defined as above.
    The {\bf symmetrization of $H$} is defined as \[  \widetilde{\Sigma}(H) = \widehat{\Sigma}(H) - \langle \widehat{\Sigma}(H) \rangle,\]
    \[
       \widehat{\Sigma}(H) = \frac{1}{2}\sum_i \left( H_i\circ \theta_i^{-1} + H_i \circ \theta_i^{-1} \circ R \right),
    \]

    It is an even function of the height, which corresponds to a function $\widetilde{\Sigma}(H):  \left[-\frac{1}{2},\frac{1}{2}\right] \rightarrow \R$.
\end{defin}

\begin{center}
    \begin{figure}[H]
        \begin{subfigure}[h]{0.45\linewidth}
        \scalebox{0.85}{\tdplotsetmaincoords{80}{110}
        \begin{tikzpicture}[scale=3,tdplot_main_coords,
             declare function={f(\s,\t)=0.3*cos(3*\s)+1*\t;}]

            \tdplotsetrotatedcoords{20}{80}{0}
                \pgfmathsetmacro{\ra}{2}
                \draw [ball color=white,very thin,tdplot_rotated_coords] (0,0,0) circle (1) ;
                \node at (0,1.3,0.55) {\Large $\varphi^t_{H_i}$};
                \node at (0,-1.25,-0.1) {\Large $S_i$};
                

                \foreach \a in {0,0.01,...,0.07}{
                    \draw [domain=-75:115, color =red!35!white,samples = 60, line width=1pt] plot ({(1-(f(\x,\a))^2)^(0.5)*cos(\x)}, {(1-(f(\x,\a))^2)^(0.5)*sin(\x)},{f(\x,\a)});
                };

                \foreach \a in {0.07,0.11,...,0.43}{
                    \draw [domain=-63:107, color =red!35!white,samples = 60, line width=4pt] plot ({(1-(f(\x,\a))^2)^(0.5)*cos(\x)}, {(1-(f(\x,\a))^2)^(0.5)*sin(\x)},{f(\x,\a)});
                };

                \foreach \a in {0.43,0.44,...,0.5}{
                    \draw [domain=-75:115, color =red!35!white,samples = 60, line width=1pt] plot ({(1-(f(\x,\a))^2)^(0.5)*cos(\x)}, {(1-(f(\x,\a))^2)^(0.5)*sin(\x)},{f(\x,\a)});
              };

                \draw [domain=-75:115, color =black,samples = 60] plot ({(1-(f(\x,0))^2)^(0.5)*cos(\x)}, {(1-(f(\x,0))^2)^(0.5)*sin(\x)},{f(\x,0)});
                \draw [domain=-75:115, color =black,samples = 60] plot ({(1-(f(\x,0.5))^2)^(0.5)*cos(\x)}, {(1-(f(\x,0.5))^2)^(0.5)*sin(\x)},{f(\x,0.5)});

                \draw [domain=-75:120,very thick,color =red!20!black,samples = 60] plot ({(1-(f(\x,0.25))^2)^(0.5)*cos(\x)}, {(1-(f(\x,0.25))^2)^(0.5)*sin(\x)},{f(\x,0.25)});

        \end{tikzpicture}}
        \end{subfigure}
        \hfill
        \begin{subfigure}[h]{0.45\linewidth}
        \scalebox{0.85}{\tdplotsetmaincoords{80}{110}
        \begin{tikzpicture}[scale=3,tdplot_main_coords,sphere segment/.style args={%
            phi from #1 to #2 and theta from #3 to #4 and radius #5}{insert path={%
             plot[variable=\x,smooth,domain=#2:#1] 
             (xyz spherical cs:radius=#5,longitude=\x,latitude=#3)
             -- plot[variable=\x,smooth,domain=#3:#4] 
             (xyz spherical cs:radius=#5,longitude=#1,latitude=\x)
             --plot[variable=\x,smooth,domain=#1:#2] 
             (xyz spherical cs:radius=#5,longitude=\x,latitude=#4)
             -- plot[variable=\x,smooth,domain=#4:#3] 
             (xyz spherical cs:radius=#5,longitude=#2,latitude=\x)}},
             declare function={f(\s,\t)=0.3*cos(4*\s)+1.7*\t;}]
            \tdplotsetrotatedcoords{20}{80}{0}
                \draw [ball color=white,very thin,tdplot_rotated_coords] (0,0,0) circle (1) ;
                \draw[thin,fill=red!40!white,fill opacity=0.6,
                sphere segment={phi from -20 to 160 and theta from 17.46 to 44.5  and radius 1}] ;

                \draw [domain=-75:115] plot ({0.953*cos(\x)}, {0.953*sin(\x)},0.3);
                \draw [dashed,domain=115:285] plot ({0.953*cos(\x)}, {0.953*sin(\x)},0.3);
                \draw [domain=-75:115] plot ({0.714*cos(\x)}, {0.714*sin(\x)},0.7);
                \draw [dashed,domain=115:285] plot ({0.714*cos(\x)}, {0.714*sin(\x)},0.7);

                \draw[thin,fill=red!40!white,fill opacity=0.6,
                sphere segment={phi from -20 to 160 and theta from -17.46 to -44.5  and radius 1}] ;
                \draw [domain=-75:115] plot ({0.953*cos(\x)}, {0.953*sin(\x)},-0.3);
                \draw [name path=A,dashed,domain=115:285] plot ({0.953*cos(\x)}, {0.953*sin(\x)},-0.3);
                \draw [domain=-75:115] plot ({0.714*cos(\x)}, {0.714*sin(\x)},-0.7);
                \draw [name path=B,dashed,domain=115:285] plot ({0.714*cos(\x)}, {0.714*sin(\x)},-0.7);

                \draw [very thick,color =red!20!black, domain=-71:111] plot ({0.866*cos(\x)}, {0.866*sin(\x)},-0.5);
                \draw [very thick,color =red!20!black, domain=-71:111] plot ({0.866*cos(\x)}, {0.866*sin(\x)},0.5);

                \node at (0,1.3,0.55) {\Large $\varphi^t_{H_i\theta_i^{-1}}$};
                \node at (0,1.3,-0.55) {\Large $\varphi^t_{H_i\theta_i^{-1}R}$};

                \node at (0,-1.2,0.45) {\Large $\theta_i(S_i)$};
                \node at (0,-1.35,-0.7) {\Large $R(\theta_i(S_i))$};

        \end{tikzpicture}}
        \end{subfigure}
    \end{figure}
\end{center}

This definition does not depend on the specific choices of $\theta_i.$

\begin{proof}[Proof of Theorem \ref{thm: fqm0 symm}]
Definition \ref{def: fqm0 symm} evidently satisfies the \emphref{prop:AlgebraicPropertiesFQM} property of Theorem \ref{thm: fqm0 symm}. In order to show the \emphref{prop:HoferControlFQM} property of the same theorem, consider the following.

Since all $H_i\circ \theta_i^{-1} , H_i \circ \theta_i^{-1} \circ R $ are functions of the height we can decompose $\varphi_{\widetilde{\Sigma}(H)}^t$ to the commuting flows
\[
    \varphi_{\widetilde{\Sigma}(H)}^t = \prod_i \varphi_{H_i \circ \theta_i^{-1}}^{\frac{t}{2}}\varphi_{H_i \circ \theta_i^{-1} \circ R}^{\frac{t}{2}}.
\]
Now, by inductively using the triangle inequality, together with the bi-invariance of Hofer's metric, one gets 
\[
d\left(\varphi_H^t,\varphi_{\widetilde{\Sigma}(H)}^t\right) = d \left(\prod_i \varphi_i^t,\prod_i \varphi_{H_i \theta_i^{-1}}^{\frac{t}{2}}\varphi_{H_i \theta_i^{-1} R}^{\frac{t}{2}}\right) \leq \sum_i d\left(\varphi_{H_i}^t, \varphi_{H_i \theta_i^{-1}}^{\frac{t}{2}}\varphi_{H_i \theta_i^{-1} R}^{\frac{t}{2}}\right).
\]
Set $\norm{\psi}= d(\id,\psi)$ and $\phi_i = \varphi_{H_i}^{t/2}$,
and using the well-known fact that $\norm{[A,B]}= \norm{ABA^{-1}B^{-1}} \leq 2\min\left\{\norm{A},\norm{B}\right\}$, we have
\SSplit{
    d\left(\varphi_{H_i}^t, \varphi_{H_i \theta_i^{-1}}^{\frac{t}{2}}\varphi_{H_i \theta_i^{-1} R}^{\frac{t}{2}}\right) &=
    d\left(\left(\varphi_{H_i}^{\frac{t}{2}}\right)^2, \theta_i\varphi_{H_i}^{\frac{t}{2}} \theta_i^{-1}R^{-1}\theta_i
    \varphi_{H_i}^{\frac{t}{2}} \theta_i^{-1}R
    \right)\\
    &=d\left(\phi_i^2, \theta_i\phi_i \theta_i^{-1}R^{-1}\theta_i\phi_i \theta_i^{-1}R 
    \right) \\ 
    &\leq d\left(\phi_i, \theta_i\phi_i \theta_i^{-1}\right) + d\left( \phi_i, R^{-1}\theta_i\phi_i \theta_i^{-1}R \right) \\
    &=\norm{\phi_i^{-1}\theta_i\phi_i \theta_i^{-1}}
    +\underbrace{\norm{R^{-1}\theta_i\phi_i \theta_i^{ -1}R\phi_i^{-1}}}_{\norm{\left[R^{-1}\theta_i,\phi_i\right]}}\\
    &\leq 2\norm{\theta_i} + 2\norm{R^{-1}\theta_i} \leq 4\norm{\theta_i} + 2\norm{R}.
}
Thus, for all $t \in \R$,
\[
\abs{d(\id, \varphi_{H}^t) - d\left(\id,\varphi_{\widetilde{\Sigma}(H)}^t\right)} \leq d\left(\varphi_{H}^t,\varphi_{\widetilde{\Sigma}(H)}^t\right)\leq\sum_i \left(4\norm{\theta_i} + 2\norm{R}\right) =: C(H).
\]
\end{proof}

\subsection{Symmetrization of continuous functions}\label{subsec: symm qm}
We now want to define $\Sigma$ as a map $\Sigma:C_0^0(\S^2)\rightarrow C_{\text{even},0}^0\left(I\right).$ Clearly the set $\FQM_0(\S^2)$ defined in Section \ref{subsec: symm fqm} is a $C^0$-dense subset of $C_0^0(\S^2).$ Thus a natural way to define $\Sigma$ would be via continuity. Indeed, set $I_k = \left[-\frac{1}{2}+ \frac{1}{k+1},\frac{1}{2}- \frac{1}{k+1}\right]$ for $k \geq 2$, and let $r_k : C_{\text{even}}^0\left(I\right) \rightarrow C_{\text{even}}^0\left(I_k\right)$ be the restriction map. We will prove the following result.
\begin{thm}\label{thm: Lip}
There exists a unique map $\Sigma:C_0^0(\S^2)\rightarrow C_{\text{even},0}^0\left(I\right)$ that coincides with $\widetilde{\Sigma}$ on $\FQM_0(\S^2)$, such that the map $r_k \circ \Sigma$ is $(k-1)$ Lipschitz.
\end{thm}

The proof of this statement requires the {\bf link quasi-morphisms} $\mu_{k,B}: \Ham(\S^2) \to \R$ from \cite{CHMSS} (cf. invariants introduced in \cite{PSh}). We proceed as follows. Fix $k>1$ and $B \in \left(\frac{1}{k+1},\frac{1}{2}\right)$ and set $C=\frac{1 -2B}{k-1}$.
Let $H \in \FQM_0(\S^2).$ Recall the commutative-additivity of $\mu_{k,B}$ (see \cite{CHMSS}) and consider $\left\{\theta_i\right\} \subset \Ham(\S^2)$ from the definition of $\widetilde{\Sigma}.$ Since both $\left\{\varphi_{H_i}\right\}$ and $\left\{\theta_i\varphi_{H_i}\theta_i^{-1}\right\}$ are collections of commuting maps, and as $\widetilde{\Sigma}(H)$ is a map of the height, we have
\SSplit{
    \mu_{k,B}(\varphi_H)   &= \mu_{k,B}\left(\prod_i \varphi_{H_i}\right)= \sum_i \mu_{k,B}(\varphi_{H_i}) = \sum_i \mu_{k,B}(\theta_i\varphi_{H_i}\theta_i^{-1})\\ &= \mu_{k,B}\left(\varphi_{\widetilde{\Sigma}(H)}\right)= \frac{1}{k}\sum_{j=0}^{k-1} \widetilde{\Sigma}(H)\left(-\frac{1}{2}+B+jC\right).}
    
Consider the set $L_{\ell,B}:= \left\{-\frac{1}{2} + B + jC\right\}^{\ell-1}_{j=0} \subset I,$ where $1/(\ell+1) < B <1/2.$ Then $L_{k,B}$ and $L_{k-2,B+C}$ coincide on all but two points $\pm \left( -\frac{1}{2} + B \right).$ (Note that indeed $1/(k-1)<B+C<1/2.$) This gives us the following explicit expression for $\widetilde{\Sigma}:$
\begin{equation}\label{SymmExplicit}
    \widetilde{\Sigma}(H)(z) = \frac{k}{2}\mu_{k,B}(H) -    \frac{k-2}{2}\mu_{k-2,B+C}(H), \quad
    \stackL{
        \abs{z} < \frac{1}{2}-\frac{1}{k+1},\\
        B= \frac{1}{2}-z,\\
        C=\frac{1 -2B}{k-1}
    }.
\end{equation}

We can now use the Hofer-Lipschitz property of $\mu_{k,B}$ to deduce our Lipschitz condition. For all $H,H' \in \FQM_0(\S^2)$ 
\SSplit{
    \abs{\Sigma(H')(z) - \Sigma(H)(z)} &=
    \abs{\frac{k}{2}\mu_{k,B}(H') -    \frac{k-2}{2}\mu_{k-2,B+C}(H')-\frac{k}{2}\mu_{k,B}(H) +    \frac{k-2}{2}\mu_{k-2,B+C}(H)}\\
    &\leq
    \frac{k}{2}\abs{\mu_{k,B}(H') -\mu_{k,B}(H) }+
    \frac{k-2}{2}\abs{\mu_{k-2,B+C}(H')- \mu_{k-2,B+C}(H)}\\
    &\leq \left(\frac{k}{2} + \frac{k-2}{2} \right) \norm{H'-H}_{C^0} = (k-1)\norm{H'-H}_{C^0}.
}

This proves the Lipschitz-continuity condition on $\widetilde{\Sigma}$ and finishes the proof of the theorem. 

\begin{rem}\label{rem:deltas} It should be noted that the explicit formula (\ref{SymmExplicit}) directly extends to $\Sigma$, and inherits the independence of $k$ from $\widetilde{\Sigma}$. The latter can also be shown directly as follows. For $k, B$ as above, let $\delta_{k,B} = \frac{1}{k} \sum_{x \in L_{k,B}} \delta_x.$ If $\delta_i = \delta_{k_i, B_i}$ satisfy $\sum a_i \delta_i = 0$ for certain $a_i \in \R,$ then the Lagrangian control property yields that the quasi-morphism $\mu=\sum a_i \mu_{k_i, B_i}$ vanishes on $\cl T$ and $\widehat{\cl T}.$ This immediately implies the independence of the above formula for $\Sigma(H)(z)$ on $k$ as long as $|z|<1/2-1/(k+1).$ (We also note that the values $k=1$ and $B=1/2$ are admissible for this argument.)
\end{rem}

\medskip
\noindent The map $\Sigma:C_0^0(\S^2)\rightarrow C_{\text{even},0}^0\left(I\right)$ inherits the algebraic properties of $\widetilde{\Sigma},$ satisfies Lipschitz continuity, and a new property: independence of Hamiltonian, as described below. 

Recall that to every Hamiltonian $H$ in $C_0^0(\S^2)$ there corresponds a one-parameter subgroup $\phi_H = \{\phi^t_H\}$ in $\wh{\Ham}(\S^2).$ It is obtained by approximating $H$ by smooth $H_i$ and considering the limit $\phi^t_H$ in $\wh{\Ham}(M)$ of the Hofer Cauchy sequence $\phi^t_{H_i}.$ It can be checked that $\phi^t_H = \phi^1_{tH}.$ We define $\Aut^0(M)$ as the image of the map $C_0^0(\S^2) \to \wh{\Ham}(\S^2),$ $H \mapsto \phi^1_H.$  

\begin{rem}
A related subspace $\widehat{\Aut}(\S^2)$ of $\wh{\Ham}(\S^2)$ is defined as the closure inside $\wh{\Ham}(\S^2)$ of $\Aut(\S^2).$ In other words, it consists of the limits in $\wh{\Ham}(\S^2)$ of all Hofer Cauchy sequences $\phi_i$ in $\Aut(\S^2).$ Clearly $\Aut^0(\S^2) \subset \wh{\Aut}(\S^2).$ It would be very interesting to study the extent to which this inclusion is strict, as $\wh{\Aut}(\S^2)$ is known to contain interesting elements: for instance the Anosov-Katok pseudo-rotations. 
\end{rem}

Recall that for $k>1$, $I_k = \left[-\frac{1}{2}+ \frac{1}{k+1},\frac{1}{2}- \frac{1}{k+1}\right]$, while $r_k : C_{\text{even}}^0\left(I\right) \rightarrow C_{\text{even}}^0\left(I_k\right)$ is the restriction map.

\begin{thm}\label{thm: prop basic} The map $\Sigma:C_0^0(\S^2)\rightarrow C_{\text{even},0}^0\left(I\right)$ satisfies the following properties.

\begin{enumerate}
    \item \emph{(Independence of Hamiltonian)} \setlabel{Independence of Hamiltonian}{prop:IndependenceOfHamiltonian}If $H,H'\in C_0^0(\S^2)$ generate the same elements $\varphi^1_{H}=\varphi^1_{H'} \in \Aut^0(M)$, 
    then $\Sigma(H)=\Sigma(H'),$ i.e.
    $\Sigma$ descends to a map from ${\Aut}^0(M) $ to $ C_{\text{even},0}^0\left(I\right).$ We denote this map by $\sigma,$ i.e. $\sigma(\varphi_H^1) = \Sigma(H)$.
    \item \emph{(Lipschitz continuity)}  \setlabel{Lipschitz continuity}{prop:LipschitzContinuity}
    Both symmetrization maps $\Sigma$ and $\sigma$ are $(k-1)$ Lipschitz after restriction to $I_k.$ That is, $r_k\circ \sigma : \left({\Aut}^0(M),\widehat{d}\right) \rightarrow \left(C^0_{\text{even}}(I_k),d_{C_0}\right)$, and $r_k\circ \Sigma : \left(C_0^0(\S^2),{d}_{C_0}\right) \rightarrow \left(C^0_{\text{even}}(I_k),d_{C_0}\right)$ are $(k-1)$ Lipschitz.
    \item \emph{(Algebraic properties)} \setlabel{Algebraic properties}{prop:AlgebraicProperties}For all $H\in C_0^0(\S^2)$ and $t \in \R$
    \[
    \Sigma(tH) = t\Sigma(H).
    \]
    If Hamiltonians $F,G$ have commuting flows, then
    \[
    \Sigma(F) + \Sigma(G) = \Sigma(F+G),
    \]
    and if $H(z)$ is a function of the height, then
    \[
        \Sigma(H) =  \frac{H(z) + H(-z)}{2}.
    \]
    
\end{enumerate}
\end{thm}

\begin{proof}
The only new point to prove is independence of Hamiltonian, which follows on smooth functions by its expression \eqref{SymmExplicit} in terms of the link quasi-morphisms, as well as the fact that these quasi-morphisms vanish on $\pi_1(\Ham(S^2)) \cong \Z/2\Z.$ This extends to continuous functions by Lipschitz continuity of the link quasi-morphisms in Hofer's metric.
\end{proof}

In the course of our arguments, it will sometimes be convenient to extend the symmetrization of functions to the whole $C^0(S^2).$ The following immediate lemma shows that we may do this.

\begin{lemma}\label{lma: no zero mean}
The map $\Sigma:C_0^0(\S^2)\rightarrow C_{\text{even},0}^0\left(I\right)$ extends to a map $\Sigma:C^0(\S^2)\rightarrow C_{\text{even}}^0\left(I\right)$ by setting \[\Sigma(F+c) = \Sigma(F)+c\] for all $F \in C^0_0(S^2)$ and $c \in \R.$
The extended map satisfies the \ref{prop:AlgebraicProperties} for $F, G, H \in C^0(S^2)$ as well as the \ref{prop:LipschitzContinuity} property in the sense that $r_k \circ \Sigma$ is $(k-1)$ Lipschitz on $C^0(S^2).$
\end{lemma}

We finish this section with two remarks of general interest.

\begin{rem}
It is not hard to see that the \ref{prop:LipschitzContinuity} property of Theorem \ref{thm: prop basic} implies that $\sigma$ extends to $\wh{\Aut}(\S^2).$ It would be very interesting to see if further results in this paper extend to this setting. 
\end{rem}

\begin{rem} \label{remark:hat-Ham-non-simple}
Motivated by the recent proof of non-simplicity of the groups of Hamiltonian homeomorphisms on surfaces \cite{CHMSS}, it is natural to ask whether 
the group $\wh{\Ham}(\S^2)$ is simple or not. This question was communicated to us by Patrick Foulon. Let us explain why $\wh{\Ham}(\S^2)$ is not a simple group. Consider the quasimorphism $ \mu_k := \mu_{k,B_0} $ where $ B_0 =  \frac{1}{k+1} \Area(\S^{2}) $. It was shown in \cite{CHMSS} that $ \mu_k $ is a quasimorphism with defect $ \frac{2}{k} $. Moreover, by the Hofer continuity, $ \mu_k $ extends to a quasimorphism on $\wh{\Ham}(\S^2)$, with the same defect $ \frac{2}{k} $. Motivated by \cite{CHMSS2}, we define $ G = \{ \psi \in \wh{\Ham}(\S^2) \, | \, k\mu_k(\psi) = O(1) \text{ when } k \rightarrow \infty \} $. Then $ G $ is a normal subgroup of $\wh{\Ham}(\S^2)$, and by \cite[Theorem 3.1]{CHMSS2} it contains $ Ham(\S^2) $. To see that $ G $ is a proper subgroup of $\wh{\Ham}(\S^2)$, notice that the proof of \cite[Second item of Theorem 1.2]{CHMSS2} (see \cite[Section 4.3]{CHMSS2}) constructs an element of $ \Hameo(S^2) $ (hence of $\wh{\Ham}(\S^2)$) which in particular does not lie in $ G $. Therefore $\wh{\Ham}(\S^2)$ is not a simple group.
\end{rem}

\subsection{A combinatorial point of view on the symmetrization}\label{subsec: combin}

We conclude this section by formalizing the observation that while $\Sigma$ is defined on $C_0^0(\S^2)$, its values on functions on the Reeb graph of a given Morse function admit a combinatorial description.

Given a finite tree $\Gamma$, and $\mu$ a probability measure on $\Gamma$, which is Lebesgue on every edge and such that all the vertices have measure $0$, we want to define a map 
\[
    \Sigma^\Gamma: C_0(\Gamma) \to C_{\text{even},0}\left( I \right), \;\; h \mapsto \Sigma_h^\Gamma,
\]
where $C_0(\Gamma)$ is the set of mean-zero (with respect to $\mu$) functions, which is related to $\Sigma$ as follows: if $f$ is a Morse function on $S^2$ and $H = h\circ \pi_{\Gamma_f}$ is an QM function subordinate to $f$ then
$\Sigma(H) = \Sigma^{\Gamma_f}_h$. (Recall that $\pi_{\Gamma_f}:\S^2 \rightarrow \Gamma_f$ is the natural projection.) 

Let us start with some notation. Consider $\Gamma$ as a simplicial complex, denote its edges by $e = [v,w]$, where $v,w$ are vertices.
For an edge $e$ we denote by $e^\circ$ its interior. Given a vertex $v$ of $e$, denote by $T_{e,v}$ the connected component of $\Gamma \setminus e^\circ$ containing $v$.

Define the intervals
$$I^-_{e,v} = (-1/2, -1/2 + \mu(T_{e,v})] = -I^+_{e,w}\;,$$
$$I_{e,v} = \left(-1/2 + \mu(T_{e,v}),-1/2 + \mu(T_{e,v}) + \mu(e)\right) = - I_{e,w}\;,$$
$$I^+_{e,v} = [-1/2 + \mu(T_{e,v}) + \mu(e), 1/2) = - I^-_{e,w}\;.$$
\begin{defin} \label{def: elementary function}
    A function $h\in C_0(\Gamma)$ is called elementary with respect to edge $e=[v,w]$ if it is possibly non-constant on an edge $e$, and is equal to the constant $h(v)$ on $T_{e,v}$, and to the constant $h(w)$ on $T_{e,w}$.  
\end{defin}

\begin{prop}\label{prop: decomp}
Every function $h\in C_0(\Gamma)$ admits a unique decomposition into a sum $h = \sum_{e} h_e$ of elementary functions, where $e$ runs over the edges of $\Gamma.$
\end{prop}

\begin{proof}
Denote by $\widetilde{h}_e$ the function which is equal to $h$ on $e=[v,w]$, equal to $h(v)$ on $T_{e,v}$ and equal to $h(w)$ on $T_{e,w}.$ Denote
\[
h_e = \widetilde{h}_e - \int\widetilde{h}_e d\mu.
\]
Note that $a:= h - \sum_e h_e$ is constant on each edge. Therefore $a$ is constant on $\Gamma$ and by construction it has zero mean. Hence $a\equiv 0$ and the decomposition $h = \sum_e h_e$ follows. The uniqueness of such a decomposition is clear.

\end{proof}

\begin{thm}\label{thm: symm combin} There is a unique $\R$-linear map 
$\Sigma^\Gamma: C_0(\Gamma) \to C_{\text{even},0}\left( I \right)$, such that for elementary $h$, if we put
$$\Sigma_{h,v}^\Gamma(z) = h(v), \;\; \text{if}\;\; z \in I^-_{e,v}\;;$$
$$\Sigma_{h,v}^\Gamma(z) = h(x), \;\;\text{if} \;\; z = -1/2 + \mu(T_{e,v})+ \mu([v,x]) \in I_{e,v}
\;(\text{here}\; x \in e)\;;$$
$$\Sigma_{h,v}^\Gamma(z) = h(w), \;\;\text{if} \;\; z \in I^+_{e,v}\;,$$
then
$$\Sigma_h^\Gamma(z) = \frac{1}{2}\left(\Sigma_{h,v}^\Gamma(z)+\Sigma_{h,w}^\Gamma(z)\right)\;.$$
\end{thm}
\begin{proof}[Proof of Theorem \ref{thm: symm combin}]
    We have a couple of things to show. First, since
    $$\Sigma_{h,w}^\Gamma(-z)= \Sigma^\Gamma_{h,v}(z)\;,$$
    we have that $\Sigma_h^\Gamma (z) = \Sigma_h^\Gamma(-z)$.
    Second, since the map
$$e \to I_{e,v}, \;\; x \mapsto  -1/2 + \mu(T_{e,v})+ \mu([v,x]) $$
is a homeomorphism sending measure $dz$ to $\mu$, one readily checks
that
$$\int_{-1/2}^{1/2} \Sigma^\Gamma_{h,v}(z) dz = \int_\Gamma hd\mu=0.$$
\end{proof}
Note that the fact that $\Sigma^\Gamma$ coincides with $\Sigma$ in the way discussed in the beginning of this section is evident by direct comparison between the description of $\Sigma^\Gamma$ for elementary $h$ from Theorem \ref{thm: symm combin}  and Equation \eqref{eq: symm elementary} in Section \ref{subsec: uniform}. 

Finally, it is interesting to write down an explicit formula for the map $\Sigma^{\Gamma}.$ Let us consider $\Gamma$ as an oriented graph with every edge oriented in both ways. 
Let $\ar{e} =  [v,w]$ be an oriented edge and $e \in E$ the corresponding unoriented edge. Let $A = \mu(T_{e,v}),$ $B=\mu(T_{e,w}).$ Define a continuous map \[\tau_{\ar{e}}:[-1/2,1/2] \to \Gamma\] to be a measure and orientation preserving homeomorphism $[-1/2+A, 1/2-B] \to e$ extended by constant maps to the rest of the interval, as in the following diagram.

\begin{center} \begin{tikzpicture}[
        mid arrow/.style={
            postaction={decorate,decoration={
                markings,
                mark=at position .5 with {\arrow{latex}}
        }}
      },
    ]

\draw [thick] (0,5) -- (0,0);

\node at (0,5) [circle,fill,inner sep=1pt]{};
\node at (0,5) [left] {$\frac{1}{2}$};

\node at (0,0) [circle,fill,inner sep=1pt]{};
\node at (0,0) [left] {$-\frac{1}{2}$};

\node at (0,1) [circle,fill,inner sep=1pt]{};
\node at (0,1) [left] {$-\frac{1}{2}+A$};

\node at (0,4) [circle,fill,inner sep=1pt]{};
\node at (0,4) [left] {$\frac{1}{2}-B$};

\draw[red, very thick] (0,1) -- (0,2);
\node at (0,2) [circle,fill,red,inner sep=1pt]{};
\draw [decorate, decoration = {calligraphic brace, mirror},thick,pen colour=violet] (0.1,1.05) --  (0.1,1.95);
\node at (0.1,1.5) [right, violet] {$z$};

\draw (0,1) -- (3,1) [dashed];
\draw (0,4) -- (3,4) [dashed];

\draw[mid arrow,thick] (3,1) -- (3,4);
\node at (3,1) [circle,fill,inner sep=1pt]{};
\node at (3,4) [circle,fill,inner sep=1pt]{};
\node at (3,2.45) [right] {$\vec{e}$};

\draw [->,red] (0.1,5) -- (2.9,4.075);
\draw [->,red] (0.1,4.5) -- (2.7,4.05);
\draw [->,red] (0.1,0) -- (2.9,0.925);
\draw [->,red] (0.1,0.5) -- (2.7,0.95);

\draw [->,red] (0.1,2) -- (2.9,2);
\node at (1.5,2) [above, red] {$\tau_{\vec{e}}$};

\draw [decorate, decoration = {calligraphic brace, mirror},thick,pen colour=violet] (3.1,1.05) --  (3.1,1.95);
\node at (3.1,1.5) [right, violet] {$\text{Area}=\pi\cdot z$};

\draw [decorate,decoration={snake,amplitude=0.8},thick,blue] (3,1) --  (2.5,0);
\draw [decorate,decoration={snake,amplitude=0.8},thick,blue] (3,1) --  (3,0);
\draw [decorate,decoration={snake,amplitude=0.8},thick,blue] (3,1) --  (3.5,0);
\draw [decorate, decoration = {calligraphic brace, mirror},thick,pen colour=blue] (3.6,0) --  (3.6,1);
\node at (3.6,0.5) [right, blue] {$\text{Area}=\pi\cdot A$};

\draw [decorate,decoration={snake,amplitude=0.8},thick,blue] (3,4) --  (2.5,5);
\draw [decorate,decoration={snake,amplitude=0.8},thick,blue] (3,4) --  (3,5);
\draw [decorate,decoration={snake,amplitude=0.8},thick,blue] (3,4) --  (3.5,5);
\draw [decorate, decoration = {calligraphic brace, mirror},thick,pen colour=blue] (3.6,4) --  (3.6,5);
\node at (3.6,4.5) [right, blue] {$\text{Area}=\pi\cdot B$};
\end{tikzpicture} \end{center}

Denote by $\ar E$ the set of oriented edges of $\Gamma,$ $E$ the set of unoriented edges of $\Gamma$ and $V$ the set of vertices of $\Gamma.$ For each $v \in V$ let $d(v)$ denote its (unoriented) degree.

\begin{prop}\label{prop: exact formula}
For every $h \in C_0(\Gamma)$ we have \[\Sigma^{\Gamma}(h) = \frac{1}{2} \sum_{\ar e} \tau_{\ar e}^*h - \sum_{v \in V} (d(v)-1) h(v).\]
\end{prop}

Note that the sum can equivalently be taken over the interior vertices only, as the coefficient $d(v)-1$ vanishes for leaves.

\begin{proof}
One argument is as follows. First, by the proof of Proposition \ref{prop: decomp} and Theorem \ref{thm: symm combin}, \[\Sigma^{\Gamma}(h) = \frac{1}{2} \sum_{\ar e \in \ar E} \tau_{\ar e}^*h - C\] for the constant \[C = \sum_{e \in E} \int \til{h}_e  d\mu= \frac{1}{2} \sum_{\ar e \in \ar E} \int \tau_{\ar e}^*h d\mu.\] Now by direct verification \[C = \int h d\mu + \frac{1}{2} \sum_{\ar e=[v,w] \in \ar E} \left( h(v) \mu(T_{e,v}) + h(w) \mu(T_{e,w}) \right) = \sum_{\ar e=[v,w] \in \ar E} h(v) \mu(T_{e,v}).\] However, for every interior $v \in V,$ the contribution of $v$ to the sum is $\sum_{e\in E, v \in e} \mu(\Gamma_{v; e}) h(v) ,$ where $\Gamma_{v; e}$ is the connected component of $\Gamma \setminus e^\circ$ containing $v.$ Note that $\Gamma_{v; e}$ is the union of $\{v\}$ and the $d(v)-1$ connected components of $\Gamma \setminus \{v\}$ not containing $e^\circ.$ Hence \[\sum_{e \in E, v \in e} \mu(\Gamma_{v; e}) h(v) = (d(v)-1) h(v)\] and therefore \[C = \sum_{v \in V} (d(v)-1) h(v)\] as required.

\end{proof}

\begin{rem}
Another argument proving Proposition \ref{prop: exact formula} consists in decomposing $h$ as \[h = \sum_e \til{h}_e -  \sum_{v \in V} (d(v)-1) h(v),\] where $\til{h}_e$ was defined in the proof of Proposition \ref{prop: decomp}, by a different combinatorial argument. Essentially, one proves that $\sum_{e \in E, x \notin e} \til{h}_e(x) = \sum_{v \in V} (d(v)-1) h(v).$ The argument we presented in detail is closer to our previous discussion.
\end{rem}

\section{Hofer growth of subgroups}\label{sec: proof thm}

\subsection{Enhanced Dichotomy}\label{subsec: state thm}
Our ultimate goal in this section is to prove the following ``enhanced dichotomy" theorem.
\begin{thm}\label{thm: dichotomy}
    Let $H\in C^2_0(\S^2)$ be a mean-zero $C^2$ function on $S^2.$ Then:
    \begin{enumerate}
        \item If $\Sigma(H) \not\equiv 0$, then $H$ has linear growth type, i.e. $\r(H)>0$.
        \item If $\Sigma(H) \equiv 0$, then $\widehat{d}(\id,\varphi_H^t) \leq 19 \Area(\S^2)$.
    \end{enumerate}
\end{thm}

Our proof proceeds by establishing the following useful properties of the symmetrization map $\Sigma.$ Note that Theorem \ref{thm: dichotomy} is a direct consequence of the fifth property, \emphref{prop:BoundedGrowthControl}.
\begin{thm}\label{thm: prop}
    The symmetrization map $\Sigma$ has the following properties:
    \begin{enumerate}
                \item \emph{($C^2$ stability)}\setlabel{$C^2$ stability}{prop:UniformControl}
                There exists a constant $C > 0$ such that every $H\in C^2_0(S^2)$ satisfies
                \[
                \norm{\Sigma(H)}_{C^0} \leq C \norm{H}_{C^2}.
                \]
                Furthermore, for $H\in C^2_0(S^2),$ $\Sigma(H)$ extends to a continuous function on $\left[-\frac{1}{2}, \frac{1}{2}\right].$
                \item \emph{(Hölder's inequality)}\setlabel{Hölder's inequality}{prop:Holder}
                Every pair $H,H' \in C^2_0(\S^2)$ with $\norm{H-H'}_{C^0}<1$ satisfies
                \[
                \widehat{d}(\varphi_{\Sigma(H)}^1,\varphi_{\Sigma(H')}^1)    \leq C' \sqrt{\norm{H-H'}_{C^0}}\left(\sqrt{1+ \norm{H}_{C^2}+\norm{H'}_{C^2}}\right)
                \]
                for a universal constant $C'>0$.
                \item \emph{(Flattening control)}\setlabel{Flattening control}{prop:FlatteningControl}
                Let $H$ be a Morse function, $\varepsilon > 0 $, and $\widetilde{H}$ an $\varepsilon$-flattening of $H$, then
                \[
                    \norm{\Sigma(\widetilde{H}) - \Sigma(H)}_{C^0} \leqslant (1+2e(H)) \varepsilon
                \]
                where $e(H)$ is the number of edges of the Reeb graph of $ H $ (which is the same as the number of critical points of $H$ minus one).
                \item \emph{(Enhanced proximality)}\setlabel{Enhanced proximality}{prop:ConjugationControl}
                For every $H\in C^2_0(\S^2)$ there exists $\psi \in \widehat{\Ham}(M)$ such that
                \[
                    \widehat{d}(\varphi^1_{H}, \psi^{-1}\varphi^1_{\Sigma(H)}\psi) \leq 19 \Area(\S^2).
                \]
                \item \emph{(Enhanced dichotomy)}\setlabel{Enhanced dichotomy}{prop:BoundedGrowthControl}
                For every $H\in C^2_0(M)$,
                \[
                \r(H) =0 \iff \Sigma(H)\equiv 0 \iff \forall t, \quad \widehat{d}(\id,\varphi_H^t) \leq 19 \Area(\S^2).
                \]
    \end{enumerate}
\end{thm}

We note that the \emphref{prop:FlatteningControl} property of Theorem \ref{thm: prop} has the following consequence, which we will use in the proof of the \emphref{prop:BoundedGrowthControl} property.

\begin{cor}[Partial Morse proximality]\label{cor: Morse proximality}
Let $H \in C^2_0(S^2)$ be a Morse function. Then $\rho(H) = \rho({\Sigma}(H)).$
\end{cor}

\begin{proof}
Let $\eps>0.$ Note that by the \emphref{prop:UniformControl} property of Theorem \ref{thm: prop}, $\Sigma(H)$ is in $C^0_0(S^2).$ By the \emphref{prop:FlatteningControl} property there exists $H_1 \in \FQM_0(S^2)$ such that $||H-H_1||_{C^0} < \eps$  and $||\Sigma(H)-\Sigma(H_1)||_{C^0} < \eps.$ Now by the \emphref{prop:HoferControlFQM} property of Theorem \ref{thm: fqm0 symm}, $\rho(H_1) = \rho(\Sigma(H_1)).$ Furthermore, by definition $|\rho(H)-\rho(H_1)|<\eps$ and $|\rho(\Sigma(H))-\rho(\Sigma(H_1))| < \eps.$ Combining these estimates we obtain that $|\rho(\Sigma(H))-\rho(H)| < 2\eps.$ As this holds for all $\eps > 0,$ we obtain $\rho(H) = \rho({\Sigma}(H))$ as required.
\end{proof} 

\begin{rem}
Morse functions are usually assumed to be smooth, but for the
properties of the Reeb graph used in Section \ref{subsec: flatten} below, one only requires the Morse lemma, which holds for $C^2$ functions \cite[Theorem 3.1.1]{Nirenberg}. 
\end{rem}

\subsection{The Flattening control property}\label{subsec: flatten}
Here we prove the \emphref{prop:FlatteningControl} property of Theorem \ref{thm: prop}. For a Morse function $ f: S^2 \rightarrow \mathbb{R},$ denote by $ e(f) $ the number of edges of the Reeb graph $\Gamma_f$ of $f$.
\begin{claim}
If $ f: S^2 \rightarrow \mathbb{R} $ is a Morse function, and $ H : S^2 \rightarrow \mathbb{R} $ is a subordinate quasi-Morse function, then 
\begin{equation} \label{eq:osc-sigma}
\osc \,\Sigma(H)  \leqslant e(f) \osc \, H .
\end{equation}
\end{claim}
\begin{proof}
We can find a sequence $ H_i : S^2 \rightarrow \mathbb{R} $ of smooth functions which uniformly converges to $ H $, such that each $ H_i $ descends to $ \widehat{H_i}: \Gamma_f \rightarrow \mathbb{R} $ which is constant near the vertices of $ \Gamma_f $ of degree $\geq 2.$ If $ (\ref{eq:osc-sigma}) $ holds for every $ H_i $, then from the Lipschitz continuity property for $ \Sigma $ it would follow that $ (\ref{eq:osc-sigma}) $ holds also for $ H $.  Hence, without loss of generality we can assume that the induced function $ \widehat{H}: \Gamma_f \rightarrow \mathbb{R} $ is constant near all vertices of $ \Gamma_f $ of degree $\geq 2$.

We can decompose
$ H = \sum_{e \in E(\Gamma_f)} H_e $, where $ E(\Gamma_f) $ is the set of the edges of $ \Gamma_f $, and for each $ e \in E(\Gamma_f) $, $ H_e $ is smooth and descends to $ \widehat{H_e} : \Gamma_f \rightarrow \mathbb{R} $ such that $ \widehat{H_e} $ is locally constant on $ \Gamma_f \setminus \tilde e $ where $ \tilde e \subset e^{\circ} $ is a compact subset (this determines the functions $ H_e $ uniquely up to an additive constant). Notice that $ \osc \, H_e = \osc_{e} \, \widehat{H} \leqslant \osc \, H $. Moreover, notice that for each $ H_e $ there exists $ \psi_e \in \Ham(S^2) $ such that $ F_e:= \psi_e^* H_e $ depends only on the $z$-coordinate. We have $ \Sigma(H_e)(z) = \Sigma(F_e)(z) = (F_e(z)+F_e(-z))/2 $, hence $$ \osc \, \Sigma(H_e) \leqslant \osc \, F_e = \osc \, H_e \leqslant \osc \, H .$$ In addition, we have $ \Sigma(H) = \Sigma(\sum_{e \in E(\Gamma_f)} H_e) = \sum_{e \in E(\Gamma_f)} \Sigma(H_e) $.
We finally conclude
$$ \osc \, \Sigma(H) \leqslant \sum_{e \in E(\Gamma_f)} \osc \, \Sigma(H_e) \leqslant e(f) \osc \, H .$$
\end{proof}

\begin{cor}\label{cor: flattening control}
Let $ H : S^2 \rightarrow \mathbb{R} $ be a Morse function, let $ \varepsilon > 0 $, and let $ r : \mathbb{R} \rightarrow \mathbb{R} $ be a smooth function such that $ |r(t) - t| \leqslant \varepsilon $ for every $ t \in \mathbb{R} $. 
Denote $ \widetilde{H} = r \circ H $. Then $ \| \Sigma(\widetilde{H}) - \Sigma(H) \|_{C^0} \leqslant (1+2e(H)) \varepsilon $.
\end{cor}
\begin{proof}
As $\til{H},$ $H$ commute, in the sense that their Poisson bracket satisfies $\{\til{H}, H \} = 0,$ we have $ \Sigma(\widetilde{H}) - \Sigma(H) = \Sigma(\widetilde{H}-H).$ Moreover 
$$ \left| \int_{S^2} \Sigma(\widetilde{H}-H)\omega \right| =  \left| \int_{S^2} (\widetilde{H}-H) \omega \right| \leqslant \varepsilon \Area(S^2).$$ Hence, by the lemma 
\begin{equation*}
\begin{gathered}
 \| \Sigma(\widetilde{H}) - \Sigma(H) \|_{C^0} = \| \Sigma(\widetilde{H} - H) \|_{C^0} \leqslant \varepsilon + \osc \, \Sigma(\widetilde{H} - H) \\
 \leqslant \varepsilon + e(H) \osc \, (\widetilde{H} - H) \leqslant \varepsilon + 2e(H) \varepsilon = (1+2e(H))\varepsilon .
\end{gathered}
\end{equation*} \end{proof}

\subsection{The \ref{prop:UniformControl} property}\label{subsec: uniform}

In this section we will prove the \emphref{prop:UniformControl} property of Theorem \ref{thm: prop}, and show that it implies a part of the \emphref{prop:BoundedGrowthControl} property of the same theorem.\\

We will use the following immediate inequality. Let $F \in C^0(S^2)$ be a continuous function. Set $\langle F \rangle = \frac{1}{\Area(S^2)} \int_{S^2} F \om$ for the mean of $F.$ Observe that $F-\langle F \rangle$ has mean zero. Then \begin{equation}\label{ineq: no mean} ||F-\langle F \rangle||_{C^0} \leq 2 ||F||_{C^0}.\end{equation} This implies that \[\rho(F) \leq d(\phi^1_F, \id) \leq 2 ||F||_{C^0},\] which for $F$ with zero mean improves to \[\rho(F) \leq d(\phi^1_F, \id) \leq ||F||_{C^0}.\] 

An important tool in this and subsequent sections is the ``improved Sikorav trick" \cite[Lemma 2.1]{Buh} (see also \cite{Sik}), which we repeat here for the reader's convenience.

\begin{lemma}\label{lma: Sik}
    Let $(M,\omega)$ be a closed and connected symplectic surface. Let $\varepsilon>0,$ let $m$ be a positive integer, let $\mathcal{D}_0,...,\mathcal{D}_m \subset M$ be disjoint topological open disks of area $\varepsilon$ each, and let $\phi_j: \mathcal{D}_0 \rightarrow \mathcal{D}_j$ for all $1\leq j \leq m$ be symplectic diffeomorphisms. Consider $f_0,...,f_m\in \Ham(M,\omega)$ with $\supp(f_j) \subset \mathcal{D}_j.$ Define $\Phi,\Phi'\in \Ham(M,\omega)$ by
    \[
    \Phi = f_0f_1\cdots f_m
    \]
    and
    \[
        \Phi' = f_0 \prod_{j=1}^m \phi_j^*f_j
    \]
    where $\phi_j^*f_j$ is given by $\phi_j^*f_j = \left(\phi_j\right)^{-1}f_j\phi_j$ on $\mathcal{D}_0$ and $\phi_j^*f_j = \id$ on $M\setminus \mathcal{D}_0$. Then
    \[
    d(\Phi,\Phi') < 3\varepsilon.
    \]
\end{lemma}
In this section, we will mainly use the following consequence of Lemma \ref{lma: Sik}. Let $F \in C^{\infty}(\S^2)$ be supported in a disk $\mathcal{D}$ of area $B < \frac{1}{k}.$ Record the following identity $\varphi_F^1 = (\varphi_{F}^{\frac{1}{k}})^k = \left(\varphi_{\frac{1}{k}F}^1\right)^k.$ Let $\left\{\phi_i\right\}_{i=1}^k$ be symplectic diffeomorphisms taking $\mathcal{D}$ to $k$ pairwise disjoint disks of area $B.$ Then $\Phi' = \prod_i \phi^*_i\varphi_{\frac{1}{k}F}^1$ is generated by $F' = \sum_{i=1}^k \frac{1}{k} F\circ \phi_i.$ Hence by \eqref{ineq: no mean} \[ d(\id,\Phi') \leq 2\norm{F'}_{C^0} \leq \frac{2}{k}\norm{F}_{C^0}.\] Lemma \ref{lma: Sik} for $\Phi = \varphi^1_F$ implies that \begin{equation}\label{eq: Sik est} d(\id,\varphi_F^1 ) \leq d(\id,\Phi') +3 \frac{1}{k+1} \leq \frac{2}{k}\norm{F}_{C^0}+3 \frac{1}{k+1}  \leq \frac{2\norm{F}_{C^0}+3}{k}. \end{equation}
Let us state the following lemma (the \emphref{prop:UniformControl} property of Theorem \ref{thm: prop}):
\begin{lemma}
    \label{C2BoundLemma}
    Let $H\in C_0^2(M)$, then
    \[
    \norm{\Sigma(H)}_{C^0} \leq C \norm{H}_{C^2}
    \]
    for a constant $C>0$ independent of $H$. In particular, $\Sigma(H)$ is a bounded function in $C^0\left(I \right).$
\end{lemma}

\label{AreaCoords}  
\begin{proof}[Proof of Lemma \ref{C2BoundLemma}]
We prove the lemma in the following two steps. 

{\em Step 1.} We first show that it is enough to prove the $C^0$ bound for $H$ Morse with distinct critical values as long as the constant $C$ does not depend on $H$. Let $F \in C_0^2(M)$ be arbitrary. Let $\delta > 0$ and fix $k > 2.$ Let $H\in C_0^\infty(M)$ be Morse with distinct critical values such that $\norm{H-F}_{C^2} < \delta/k.$ By the \ref{prop:LipschitzContinuity} property of Theorem \ref{thm: prop basic} we then have $|\Sigma(H) - \Sigma(F)|_{C^0(I_k)} < \delta$ and therefore \[ |\Sigma(F)|_{C^0(I_k)} \leq C \norm{F}_{C^2}+(1+C/k) \delta.\] Since $\delta>0$ and $k>2$ were arbitrary, and $C$ is independent of $H,$ we obtain as desired \[ |\Sigma(F)|_{C^0(I)} \leq C \norm{F}_{C^2}.\] 


{\em Step 2.} Now assume that $H$ is Morse with distinct critical values. Let $U_i$ be the disks/annuli corresponding to edges of the Reeb graph $\Gamma_H$ of $H$ as in Section \ref{subsec: symm fqm}. Let $H_i$ be the extension by constants of $H|_{U_i}$ and let $a_i\leq b_i$ be constants such that $\im (H_i) = \overline{H(U_i)} = [a_i,b_i].$ Furthermore, denote
    \[
    \stackL{H_i' = H_i - \lrang{H_i} := H_i - \frac{1}{\Area(\S^{2})}\int_{\S^{2}} H_i, \\a_i' := a_i - \lrang{H_i},\\b_i' := b_i - \lrang{H_i}.\\}
    \]
Finally, Proposition \ref{prop: decomp} implies that \[H = \sum_i H'_i.\]    
    
Set $B(x) = \Area\left(\left\{H_i' < x\right\}\right),$ $A_i = \Area\left(\left\{H_i' \leq a_i'\right\}\right)$ and $B_i = \Area\left(\left\{H_i' \leq b_i'\right\}\right).$ Then \[ \Sigma(H'_i) = \frac{1}{2} ( \wh{\Sigma}(H'_i) + \wh{\Sigma}(H'_i)\circ R),\] where
    \begin{equation}\label{eq: symm elementary}
    \wh{\Sigma}(H_i')\left(-\frac{1}{2}+B(x)\right)  = \begin{cases}
        a_i' & B(x)<A_i\\
        x & A_i < B(x)<B_i\\
        b_i' & B(x)>B_i
    \end{cases}
    \end{equation}\\

  \begin{center}
     \tdplotsetmaincoords{80}{110}
        \begin{tikzpicture}[scale=2,tdplot_main_coords,sphere segment/.style args={%
            phi from #1 to #2 and theta from #3 to #4 and radius #5}{insert path={%
             plot[variable=\x,smooth,domain=#2:#1] 
             (xyz spherical cs:radius=#5,longitude=\x,latitude=#3)
             -- plot[variable=\x,smooth,domain=#3:#4] 
             (xyz spherical cs:radius=#5,longitude=#1,latitude=\x)
             --plot[variable=\x,smooth,domain=#1:#2] 
             (xyz spherical cs:radius=#5,longitude=\x,latitude=#4)
             -- plot[variable=\x,smooth,domain=#4:#3] 
             (xyz spherical cs:radius=#5,longitude=#2,latitude=\x)}}]
            \tdplotsetrotatedcoords{20}{80}{0}
                \draw [ball color=white,very thin,tdplot_rotated_coords] (0,0,0) circle (1) ;
                \draw[thin,fill=red!40!white,fill opacity=0.6,
                sphere segment={phi from -20 to 160 and theta from 0 to 44.5  and radius 1}] ;

                \draw [domain=-75:115] plot ({1*cos(\x)}, {sin(\x)},0);
                \draw [dashed,domain=115:285] plot ({cos(\x)}, {sin(\x)},0);
                \draw [domain=-75:115] plot ({0.714*cos(\x)}, {0.714*sin(\x)},0.7);
                \draw [dashed,domain=115:285] plot ({0.714*cos(\x)}, {0.714*sin(\x)},0.7);

                \node at (0,-1.3,-0.1) {$a_i$};
                \node at (0,-1.05,0.65) {$b_i$};
                \node [color = blue] at (0,-1.25,0.4) {$H_i(x)$};

                \node at (0,1.4,0.3) {$U_i$};
                \draw [color=black] (0,1.2,0) -- (0,1.2,0.7);
                \draw [color=black] (0,1.2,0) -- (0,1.1,0);
                \draw [color=black] (0,1.2,0.7) -- (0,1.1,0.7);

                \draw [color=blue,domain=-75:115] plot ({0.866*cos(\x)}, {0.866*sin(\x)},0.5);

                \node [color=blue] at (0,2,-0.15) {$B(x)$};
                \draw [color=blue] (0,1.66,-1) -- (0,1.66,0.6);
                \draw [color=blue] (0,1.66,-1) -- (0,1.56,-1);
                \draw [color=blue] (0,1.66,0.6) -- (0,1.56,0.6);
        \end{tikzpicture}
 \end{center}

Whenever $B(x) \in (A_i,B_i),$ by \cite[Formula (3.3)]{izosimov2016coadjoint}, \[\pd{B}{x} = T(x), \] where $T(x)$ is the period of $\left\{H'_i = x\right\}$ as a trajectory of the flow of $X_{H_i}$. A result from  \cite{Yorke_1969} gives us the existence of a constant ${C_1}>0$ depending only on the geometry of the sphere such that the following bound holds: \[ T(x) \geq \frac{{C_1}}{\norm{H}_{C^2}}.\] So by \eqref{eq: symm elementary},
    \[
        \abs{\pd{}{B}\wh{\Sigma}(H_i') } = \abs{\frac{1}{\pd{B}{x}}} \leq C_2 \norm{H}_{C^2},\qquad C_2 = \frac{1}{C_1}>0.
    \]
Thus, $\wh{\Sigma}(H'_i)$ is $L$-Lipschitz on $(A_i,B_i)$ with $L \leq C_2 \norm{H}_{C^2}.$ Recall that the oscillation of a function $G \in C^0(S^2)$ is defined as $\osc(G) = \max(G) -\min(G).$ Now from our normalization, $0\in (a'_i,b'_i)$ and hence
    \SSplit{
        \abs{\Sigma(H)} &\leq   \sum_i\max \abs{\wh{\Sigma}(H_i')} \leq \sum_i \osc(H_i') =\sum_i \osc(H_i)\\ &\leq \sum_i L \Area(U_i) = L \Area(S^2)
         \leq  C \norm{H}_{C^2},
    } for $C = C_2 \Area(S^2).$
\end{proof}


\begin{rem}
A different proof based on an upper bound on the Banach indicatrix (see \cite{PS07,PPS}) allows one to replace the $C^2_0(S^2)$ function space by the Sobolev $W^{2,2}_0(S^2)$ function space, including the norms. See the proof of Lemma \ref{lem:banach-ind-bound}. In this context, the proof we presented above in fact proves an upper bound of the Banach indicatrix in terms of the $C^2$-norm of the function. It thus reproves a weaker version of the bound in \cite{PS07} in a new way. Therefore we chose to include it in this paper. \end{rem}

Let us now show that Lemma \ref{C2BoundLemma} implies an important part of the \emphref{prop:BoundedGrowthControl} property of Theorem \ref{thm: prop}.
\begin{prop}\label{cor:iff2}
For a function $H\in C^2_0(\S^2)$, $\Sigma(H) = 0$ if and only if $\r(\varphi_H)=0.$
\end{prop}


\begin{proof}[Proof of Proposition \ref{cor:iff2}]
If $\Sigma(H) \not\equiv 0$ then $\mu_{k,B}(H)\neq 0$ for some $k,B$, and from the independence of Hamiltonian and Hofer-Lipschitz properties of $\mu_{k,B}$ (see \cite{CHMSS, PSh}) we have
\[
0 < \abs{\mu_{k,B}(H)}    \leq \r(H),
\]
where by abuse of notation $\r(H)$ denotes the asymptotic Hofer growth of $\{\phi^t_H\}$ inside $\wh{\Ham}(S^2).$

For the other direction, assume by contradiction that both $\r(\varphi_H) = 2\varepsilon>0$ and $\Sigma(H)=0$.
Let $H_k$ be a sequence of Morse functions with $\norm{H-H_k}_{C^2}< \min\left(\frac{1}{k^2},\varepsilon\right)$.\\
First, we have
\[
\abs{\frac{d(\id,\varphi^t_{H})}{t}-\frac{d(\id,\varphi^t_{H_k})}{t}} \leq \abs{\frac{d(\varphi_H^t,\varphi^t_{H_k})}{t}} < \frac{t\varepsilon}{t} = \varepsilon \implies \r(\varphi_{H_k})>\varepsilon.
\]
Second, for an interval $J \subset I$ write $J^c = I \setminus J,$ and define $\Sigma(H_k)^-,\Sigma(H_k)^+$ as
\[
    \Sigma(H_k)^-|_{I_k} = \Sigma(H_k), \qquad \Sigma(H_k)^+|_{I_k^c} =\Sigma(H_k)
\]
extended by constants. Recall that $r_k$ denotes the restriction of a function on $I$ to $I_k.$

By \emphref{prop:LipschitzContinuity}  of Theorem \ref{thm: prop basic} we have
\[
    \norm{\Sigma(H_k)^- }_{C^0} =\norm{r_k\circ \Sigma(H_k) }_{C^0}\leq k \cdot \frac{1}{k^2} = \frac{1}{k}  \implies
    \r(\varphi_{\Sigma(H_k)^-})< \frac{2}{k}.
\]
Note that $\Sigma(H_k)^+$ is supported on two disjoint disks of area $\frac{1}{k+1}.$ Therefore we can use the improved Sikorav trick, as stated in Lemma \ref{lma: Sik}, and the \ref{prop:UniformControl} property to get, for all $k$ sufficiently large, that
\begin{equation*}
    \begin{split}
        {\r}(\varphi_{\Sigma(H_k)^+}) &\leq \frac{4\norm{\Sigma(H_k)^+}_{C^0}+6}{k} \leq \frac{4\norm{\Sigma(H_k)}_{C^0}+6}{k} \leq
        \frac{4 C\norm{H_k}_{C^2}+6}{k}\\
        &\leq   \frac{4 C(\norm{H}_{C^2} + \frac{1}{k^2})+6}{k}     
        \leq   \frac{4 C\norm{H}_{C^2}+8}{k}.
    \end{split}
\end{equation*}
In total, by Corollary \ref{cor: Morse proximality} we get that
\SSplit{
    \r(\varphi_{H_k}) &=\r(\varphi_{\Sigma(H_k)}) =
    \r(\varphi_{\Sigma(H_k)^+} \circ \varphi_{\Sigma(H_k)^-}) \leq
    \r(\varphi_{\Sigma(H_k)^+}) +
    \r(\varphi_{\Sigma(H_k)^-})\\
    &\leq
    \frac{2}{k}+\frac{4C\norm{H}_{C^2}+8}{k},
}
and for $k$ large enough this contradicts that $\r(\varphi_{H_k})>\varepsilon $. \end{proof}

Let us now prove that $\Sigma(H)$ for $H \in C^2_0(S^2)$ is continuous at the endpoints $\pm \frac{1}{2}.$ It will be convenient to denote the two-sphere by $ (M,\omega)$ in the following arguments. 


\begin{lemma} \label{lem:banach-ind-bound}
Consider a function $H \in C^2(M).$ Let $ K \subset \mathbb{R} $ be the set of critical values of $ H $, and denote $ I_H = [\min H, \max H] $. Write $ I_H \setminus K = \cup_j I_j $ as at most countable disjoint union of open intervals. For every $ I_j $ there exists some $ m_j \in \mathbb{N} $ such that for each $ t \in I_j $, the preimage $ H^{-1}(t) $ is a disjoint union of $ m_j $ embedded loops in $ M $, and the preimage $ H^{-1}(I_j) $ is a union of $ m_j $ disjoint open annuli $A_{jk}$ for $1 \leq k \leq m_j.$ Then \[\sum_{j \geq 1} \sum_{k=1}^{m_j} \underset{A_{jk}}{\osc}(H) = \sum_{j \geq 1} m_j \ell(I_j) \leqslant C \| H \|_{C^2},\] where $ \ell(I_j) $ is the length of $ I_j $.
\end{lemma}

\begin{proof}
Note that the expression $\sum_{j \geq 1} m_j \ell(I_j)$ is nothing but the Banach indicatrix (see \cite{PS07} for a discussion) \[B(1,H) = \int_{-\infty}^{\infty} \beta(c,H)\,dc\] of $H,$ where $\beta(c,H)$ is the number of connected components of $H^{-1}(c).$ However \cite[Theorem 1.5]{PS07} implies that \[B(1,H) \leq C_1(\norm{H}_{L^2}+\norm{\Delta H}_{L^2}) \leq C \norm{H}_{C^2},\] for suitable constants $C_1, C > 0,$ where $\Delta$ is the Laplace-Beltrami operator on $S^2.$ This finishes the proof.
\end{proof}

\begin{defin}
Let $ H \in C^2(M) $ and let $ \delta > 0 $. Denote by $ K $ the set of critical values of $ H $, and denote $ I_H = [\min H, \max H] $. A smooth function 
$ r : I_H \rightarrow \mathbb{R} $ is called $ (H,\delta) $-admissible if the following holds:
\begin{enumerate}
 \item $ r $ is locally constant on a neighborhood of $ K $.
 \item $ 0 \leqslant r'(t) \leqslant 1 $ for every $ t \in I_H$.
 \item $ |r(t) - t | < \delta $ for every $ t \in I_H$.
 \item $ r(\min H) = \min H $.
\end{enumerate}
\end{defin}

\begin{claim}
For every $ H \in C^2(M) $ and every $ \delta > 0 $ there exists an $ (H,\delta) $-admissible function.
\end{claim}

\begin{proof}
Denote by $ K $ the set of critical values of $ H $, then $ K $ has measure zero by Sard's theorem. Let $ I_H = [\min H,\max H] $. Choose a neighborhood $ U $ of $ K $ in $ I_H $ such that the measure of $ U $ is less than $ \delta $. Pick a smooth function $ h : I_H \rightarrow [0,1] $ such that $ \supp (h) \subset U $ and $ h=1 $ on a neighborhood of $ K $. Then the function $ r : I_H \rightarrow \mathbb{R} $ given by $ r(t) = \min H + \int_{\min H}^{t} (1-h(s)) \, ds $ satisfies all the properties.
\end{proof}

\begin{lemma} \label{lem:flattening-conv}
Let $ H \in C^2(S^2) $. Then for every $ \epsilon > 0 $ there exists $ \delta > 0 $ such that if $ r : I_H \rightarrow \mathbb{R} $ is $ (H,\delta) $-admissible then
$$ \| \Sigma(r \circ H) - \Sigma (H) \|_{L_\infty(-1/2,1/2)} < \epsilon .$$
\end{lemma}

\begin{cor}
For every $ H \in C^2(S^2) $, the function $ \Sigma(H) $ is continuous on the closed interval $ [-1/2,1/2] $. 
\end{cor}

\begin{proof}
Take a sequence $ \delta_j > 0 $ that converges to $ 0 $, and then choose a sequence $ r_j : I_H \rightarrow \mathbb{R} $ of smooth functions, where $ r_j $ is $ (H,\delta_j) $-admissible for every $ j $. Then it is not difficult to see that the functions $ \Sigma(r_j \circ H) $ are continuous on $ [-1/2,1/2] $, and by 
Lemma \ref{lem:flattening-conv}, $ \Sigma(r_j \circ H) $ uniformly converges to $ \Sigma(H) $ on $ (-1/2,1/2) $ when $ j \rightarrow \infty $. Therefore $ \Sigma(H) $ extends to a continuous function on $ [-1/2,1/2] $.
\end{proof}

\begin{proof}[Proof of Lemma \ref{lem:flattening-conv}]
Let $ H \in C^2(M) $, let $ \epsilon > 0 $, and let $ \delta >0 $ be sufficiently small. We will show that if $ r, \tilde r : I_H \rightarrow \mathbb{R} $ are
$ (H,\delta) $-admissible then 
\begin{equation} \label{eq:cauchy-prelim}
\| \Sigma(r \circ H) - \Sigma(\tilde r \circ H) \|_{L_\infty(-1/2,1/2)} < \epsilon 
\end{equation}
This would be enough for concluding the lemma. Indeed, assuming $ (\ref{eq:cauchy-prelim}) $, for a given $ \epsilon > 0 $ choose $ \delta > 0 $ such that 
$ (\ref{eq:cauchy-prelim}) $ holds. For any $ (H,\delta) $-admissible
function $ r : I_H \rightarrow \mathbb{R} $, choose a sequence $ \tilde r_l : I_H \rightarrow \mathbb{R} $ such that $ \tilde r_l $ is $ (H,\delta/l) $-admissible
for every $ l $. Choose any integer $ k > 2 $. Then by $ (\ref{eq:cauchy-prelim}) $ we in particular get 
$$ \| \Sigma(r \circ H) - \Sigma(\tilde r_l \circ H) \|_{L_\infty(-1/2+1/k,1/2-1/k)} < \epsilon .$$ 
The sequence of functions $ \tilde r_l \circ H $ uniformly converges to $ H $ when $ l \rightarrow \infty $, therefore by
the Lipschitz property, the sequence of functions $ \Sigma(\tilde r_l \circ H) $ uniformly converges to $ \Sigma(H) $ on $ (-1/2+1/k,1/2-1/k) $ when $ l \rightarrow \infty $. Hence we get $$ \| \Sigma(r \circ H) - \Sigma(H) \|_{L_\infty(-1/2+1/k,1/2-1/k)} \leqslant \epsilon ,$$
and since this holds for any $ k > 2 $, we finally conclude $$ \| \Sigma(r \circ H) - \Sigma(H) \|_{L_\infty(-1/2,1/2)} \leqslant \epsilon ,$$
and the lemma follows. 

It remains to show $ (\ref{eq:cauchy-prelim}) $. So let $ H \in C^2(M) $, let $ \epsilon > 0 $, choose $ \delta > 0 $ small enough, and let $ r, \tilde r : I_H \rightarrow \mathbb{R} $ be $ (H,\delta) $-admissible, where $ I_H = [\min H, \max H] $. Denote by $ K $ the set of critical values of $ H $, and then write
$ I_H \setminus K = \cup_j I_j $ as at most countable disjoint union of open intervals. 

Now recall Lemma \ref{lem:banach-ind-bound}. For every $ I_j $ there exists $ m_j \in \mathbb{N} $ such that for each $ t \in I_j $, $ H^{-1}(t) $ is a union of $ m_j $ embedded loops on $ S^2 $, and such that $ H^{-1}(I_j) $ is a union of $ m_j $ disjoint open annuli on $ S^2 $. Write $ I_H = [a,b] $ and
$ I_j = (a_j,b_j) $. Lemma \ref{lem:banach-ind-bound} implies that
\begin{equation} \label{eq:cor-banach-ind-bound}
\sum_j m_j (b_j-a_j) \leqslant C \|H\|_{C^2} < +\infty .
\end{equation}

\begin{claim}\label{clm:3}
If $ f : I_H \rightarrow \mathbb{R} $ is a smooth function such that $ f = 0 $ on a neighborhood of $ [a,a_j] $ and $ f=const $ on a neighborhood of $ [b_j,b] $,
then $ \| \Sigma(f\circ H) \|_{\infty} \leqslant m_j \underset{I}{\osc} \, f $.
\end{claim}
\begin{proof}[Proof of Claim \ref{clm:3}]
For each $ j $, the pre-image $ H^{-1}(I_j) $ is a union of $ m_j $ annuli $ A_{j1} , \ldots, A_{jm_j} $. For each $ 1 \leqslant j \leqslant m_j $
let $ H_{jl} : M \rightarrow \mathbb{R} $ be the smooth function given by $ H_{jl} = f \circ H $ on $ A_{jl} $ and extended by a locally constant function to $M.$ Then the functions $ f \circ H $ and $ \sum_{l=1}^{m_j} H_{jl} $ differ by a constant (their difference is locally constant on $ M $ and hence globally constant), and comparing the functions at a point outside the union $ \cup_{l=1}^{m_j} A_{jl} $, we conclude that the constant equals to some $ q_j f(b_j) $, where $ 0 \leqslant q_j \leqslant m_j $ is an integer. (It can be understood as follows: the complement to each $A_{jl}$ has exactly two components, one adjacent to the $a_j$ level, and another to the $b_j$ level; $q_j$ is the number of $b_j$ components to which the chosen point belongs.)  
Hence we can write $ f \circ H = \sum_{l=1}^{m_j} \widetilde H_{jl} $, where $ \widetilde H_{jl} = H_{jl} - f(b_j) $ for $ 1 \leqslant l \leqslant q_j $ and $ \widetilde H_{jl} = H_{jl} $ for $ q_j < l \leqslant m_j $. Then for every $ l $, $ \| \widetilde H_{jl} \|_\infty \leqslant \underset{I}{\osc} f $, and moreover $ \widetilde H_{jl} $ is elementary (as in \ref{def: elementary function}) and therefore $ \| \Sigma(\widetilde H_{jl}) \|_{\infty} \leqslant \| \widetilde H_{jl} \|_\infty \leqslant \underset{I}{\osc} f.$ Since the functions $ \widetilde H_{j1}, \ldots, \widetilde H_{jm_j} $ Poisson commute, we get
$$ \Sigma(f \circ H) = \Sigma \left(\sum_{l=1}^{m_j} \widetilde H_{jl}\right) = \sum_{l=1}^{m_j} \Sigma\left(\widetilde H_{jl}\right) $$
and hence $ \| \Sigma(f \circ H) \|_{\infty} \leqslant m_j \underset{I}{\osc} f $. This finishes the proof of the claim.
\end{proof}


Now return to our functions $ r, \tilde r : I_H \rightarrow \mathbb{R} $. Since they are $ (H,\delta) $-admissible, each of them is locally constant on a neighborhood of $ K $, hence there exists $ q \in \mathbb{N} $ such that $ r $ and $ \tilde r $ are locally constant on a neighborhood of $ I_H \setminus \cup_{j=1}^q I_j $, in particular the difference $ r - \tilde r $ is locally constant on a neighborhood of $ I_H \setminus \cup_{j=1}^q I_j $. Moreover, $ r(a) - \tilde r(a) = 0 $. Hence we can write $ r - \tilde r = \sum_{j=1}^q f_j $ where each $ f_j $ satisfies the assumptions of the claim (that is, $ f_j $ is smooth, $ f_j =0 $ on a neighborhood of $ [a,a_j] $, and $ f_j = const $ on a neighborhood of $ [b_j,b] $), and $ \underset{I_H}{\osc} f_j = \underset{I_j}{osc} (r-\tilde r) $. Since $ r \circ H $ and $ \tilde r \circ H $ Poisson commute, and since all $ f_j \circ H $ pairwise Poisson commute, we get 
$$ \Sigma(r \circ H) - \Sigma(\tilde r \circ H) = \Sigma((r - \tilde r) \circ H) = \Sigma\left( \sum_{j=1}^q f_j \circ H\right) = \sum_{j=1}^q \Sigma(f_j \circ H) .$$
Hence by the claim, 
$$ \| \Sigma(r \circ H) - \Sigma(\tilde r \circ H) \|_\infty \leqslant \sum_{j=1}^q \| \Sigma(f_j \circ H) \|_\infty 
\leqslant \sum_{j=1}^q m_j \underset{I}{\osc} f_j = \sum_{j=1}^q m_j \underset{I_j}{\osc} (r-\tilde r) .$$
The functions $ r $ and $ \tilde r $ are $ (H,\delta) $-admissible, in particular $ | r(t) - t |< \delta $ and $ |\tilde r(t)-t|<\delta $ for every $ t \in I_H $, and
hence $ |r(t) - \tilde r(t)| < 2\delta $ for every $ t \in I_H $. This implies that for every $ j $ we have $ \underset{I_j}{\osc} (r-\tilde r) < 4\delta $.
But also, since $ 0 \leqslant r'(t), \tilde r'(t) \leqslant 1 $ for every $ t \in I_H $, we get $ \underset{I_j}{\osc} (r-\tilde r) \leqslant \ell(I_j) = b_j-a_j $. Therefore $ \underset{I_j}{\osc} (r-\tilde r) \leqslant \min(4\delta,b_j-a_j) $ for every $ j $, and hence we get
$$ \| \Sigma(r \circ H) - \Sigma(\tilde r \circ H) \|_\infty \leqslant \sum_{j=1}^q m_j \underset{I_j}{\osc} (r-\tilde r) \leqslant \sum_j m_j \min(4\delta,b_j-a_j) .$$

By $ (\ref{eq:cor-banach-ind-bound})$ we choose $N$ sufficiently large so that $\sum_{j > N} m_j (b_j-a_j)< \epsilon/2$ and the finite remaining sum we supress via the choice of $\delta>0$ small enough to obtain \[\sum_j m_j \min(4\delta,b_j-a_j) < \epsilon.\] We have proved $ (\ref{eq:cauchy-prelim})$.

\end{proof}


\subsection{The H\"older inequality property}\label{subsec: Holder}

Now, let us prove the \emphref{prop:Holder} property of Theorem \ref{thm: prop}. 
Before we do that, we wish to extend inequality \eqref{eq: Sik est} to continuous functions. For $u\in C_{\text{even},0}^0(\overline{I})$ where $\overline{I}=\left[-\frac{1}{2},\frac{1}{2}\right]$, let us denote
\[
\norm{u}_k = \max_{I_{k+1}}\abs{u} + \frac{2}{k} \max_{I \setminus \mathring{I}_k} \abs{u}.
\]

\begin{prop}\label{prop:Sik C0}
For $\varphi_u^1 \in \wh{\Ham}(S^2)$ generated by $u \in C_{\text{even},0}^0(\ol{I}),$ considered as a function on $S^2,$ we have
\[\widehat{d}(\id, \varphi_u^1) \leq 2 \norm{u}_k + \frac{6}{k}.\]
\end{prop}

\begin{proof}

{\em Step 1.} We first prove this statement for a smooth function $v \in C_{\text{even},0}^{\infty}(I)$ constant near the endpoints. Using a smooth partition of unity, write $v = v_-+v_+$ for two functions $v_-,v_+$ such that \[\supp v_- \subset I_{k+1},\;\; \supp v_+ \subset I_k^c\] and \[\max |v_-| \leq \max_{I_{k+1}} |v|,\;\; \max |v_+| \leq \max_{I_{k}^c} |v|.\]

Like before, we can use the improved Sikorav trick as in estimate \eqref{eq: Sik est} to get
\[
{d}\left(\id, \varphi_v^1\right) \leq
{d}\left(\id, \varphi_{v_-}^1\right) +
{d}\left(\id, \varphi_{v_+}^1\right)  \leq 2 \norm{v}_k + \frac{6}{k}.
\]

{\em Step 2.} Take $u\in C_{\text{even},0}^0(\ol{I}).$ Let $\{v_i\}$ be a sequence of smooth functions constant near the endpoints uniformly converging to $u.$ Denote by $\psi_i$ the Hamiltonian diffeomorphism generated by $v_i,$ then on one hand by Step 1, \begin{equation}\label{eq:vi} {d}\left(\id, \varphi_{v_i}^1\right) \leq 2 \norm{v_i}_k + \frac{6}{k},\end{equation} while on the other hand $\{\psi_i\}$ converges to $\varphi_u^1$ in $\widehat{\Ham}(\S^2).$ Therefore by passing to the limit in inequality \eqref{eq:vi} we obtain the desired inequality. \end{proof}



%
%

We will use this to prove \emphref{prop:Holder} property of Theorem \ref{thm: prop}.
\begin{lemma}
    Every pair $H,H' \in C^2_0(\S^2)$ with $\norm{H-H'}_{C^0}<1$ has
    \[
    \widehat{d}\left(\varphi_{\Sigma(H)}^1,\varphi_{\Sigma(H')}^1\right)    \leq C' \sqrt{\norm{H-H'}_{C^0}}\left(\sqrt{1+ \norm{H}_{C^2}+\norm{H'}_{C^2}}\right),
    \]
    for a universal constant $C'>0$.
\end{lemma}
\begin{proof}
    Let $H,H' \in C^2_0(\S^2).$ Set $\varepsilon = \norm{H-H'}_{C^0},$ $E=\norm{H}_{C^2}+\norm{H'}_{C^2},$ $w= \Sigma(H)-\Sigma(H').$ First, by the \emphref{prop:UniformControl} property  of Theorem \ref{thm: prop} one has
    \[
    \norm{w}_{C^0} \leq \norm{\Sigma(H)}_{C^0}+\norm{\Sigma(H')}_{C^0} \leq CE,
    \]
    and by the \emphref{prop:LipschitzContinuity} property of Theorem \ref{thm: prop basic}, for every $k$,
    \[
    \max_{I_{k+1}}\abs{w} \leq (k+1)\varepsilon.
    \]
    Hence by Proposition \ref{prop:Sik C0}
    \[
    \widehat{d}(\id,\varphi_{w}^1) \leq 2\norm{w}_k + \frac{6}{k} = 2\max_{I_{k+1}}\abs{w} + \frac{4}{k} \max_{I\setminus I_k} \abs{w} +\frac{6}{k} \leq 2(k+1)\varepsilon + \frac{4}{k}CE + \frac{6}{k}     \leq 2(k+1)\varepsilon + \frac{C''(E+1)}{k},
    \]
    for $C'' =  \max \left\{6,4C\right\}$. Note that by hypothesis, the set $K= \left[\sqrt{\frac{C''(E+1)}{\varepsilon}}, 2\sqrt{\frac{C''(E+1)}{\varepsilon}}\right] \bigcap \N$ is non-empty. For an arbitrary $k\in K,$ estimating $k+1 \leq 2k$ we have
    \SSplit{
        \widehat{d}(\id,\varphi_{w}^1) &\leq 2(k+1)\varepsilon + \frac{C''(E+1)}{k}
        \leq 8\sqrt{\frac{C''(E+1)}{\varepsilon}} \varepsilon + \frac{C''(E+1)}{\sqrt{\frac{C''(E+1)}{\varepsilon}}}\\
        &=\underbrace{9\sqrt{C''}}_{:=C'}\sqrt{\varepsilon} \sqrt{E+1} = C' \sqrt{\norm{H-H'}_{C^0}}\sqrt{1+\norm{H}_{C^2}+\norm{H'}_{C^2}}
    }
    yielding our result.\\
\end{proof}

\subsection{Enhanced proximality and enhanced dichotomy}\label{subsec: thm mod fqm}
In this section we will state the ``conjugation lemma", Lemma \ref{ConjugationLemma}, and use it to prove \emphref{prop:ConjugationControl}  of Theorem \ref{thm: prop} and finish the proof of the \emphref{prop:BoundedGrowthControl} property of the same theorem. The proof of Lemma \ref{ConjugationLemma} is postponed to the next section.

\begin{lemma}\emph{(The Conjugation lemma)}\label{ConjugationLemma}
Let $H'$ be a Morse function on $S^2.$ Let $\widetilde{H}=r\circ H'$ be a sufficiently small $\varepsilon$-flattening, with $r$ constant near the median of $H'.$ Then
\[
d(\varphi^1_{\widetilde{H}},\psi^{-1}\varphi^1_{\Sigma(\widetilde{H})}\psi)
\leq 18\Area(\S^2)
\]  
for some $\psi \in {\Ham}(\S^2)$.
\end{lemma}
We will now prove that the lemma implies the following (\emphref{prop:ConjugationControl} property of  Theorem \ref{thm: prop}).
\begin{cor}
    For every $H\in C^2_0(\S^2)$ there exists $\psi \in \widehat{\Ham}(M)$ such that
    \[
        \wh{d}(\varphi^1_{H}, \psi^{-1}\varphi^1_{\Sigma(H)}\psi) \leq 19 \Area(\S^2).
    \]
\end{cor}

\begin{proof}
Let $H \in C^2_0(\S^2).$ Let $H'$ be Morse with distinct critical values such that \[\norm{H-H'}_{C^2}<\varepsilon_1<1,\] where $\varepsilon_1$ is an arbitrarily small constant which we will choose later. Let us record that
\[
d(\varphi_H^1 ,\varphi_{H'}^1 )<\varepsilon_1.
\]

By flattening control (see Corollary \ref{cor: flattening control} with $H$ replaced by $H'$), we may choose a flattening $\til{H}$ of $H'$ such that $\norm{\widetilde{H}-H'}_{C^0} \leq \varepsilon_2 = \frac{1}{k^2}$ where $k > 1+2 e(H')$  and \[|\Sigma(H')-\Sigma(\til H)|_{C^0} < \frac{1}{k}. \] This implies that \[ \widehat{d}\left(\varphi_{\Sigma(H')}^1 ,\varphi_{\Sigma(\widetilde{H})}^1\right)   \leq \frac{1}{k}.\]



Let us also record that in this case
\[
    \widehat{d}\left(\varphi_{H'}^1 ,\varphi_{\widetilde{H}}^1\right) \leq \frac{1}{k^2}.
\]

Now take $\psi$ from Lemma \ref{ConjugationLemma}. We estimate that
\SSplit{
    \widehat{d}\left(\varphi^1_{{H}},\psi^{-1}\varphi^1_{\Sigma({H})}\psi\right) &\leq \widehat{d}\left(\varphi_H^1,\varphi_{\widetilde{H}}^1\right) +\widehat{d}\left(\varphi^1_{\widetilde{H}},\psi^{-1}\varphi^1_{\Sigma(\widetilde{H})}\psi\right)    \\
    &+\underbrace{\widehat{d}\left(\psi^{-1}\varphi^1_{\Sigma(\widetilde{H})}\psi,\psi^{-1}\varphi^1_{\Sigma(H)}\psi\right)}_{ \widehat{d}\left(\varphi_{\Sigma(\widetilde{H})}^1,\varphi_{\Sigma(H)}^1\right)}.
}
Furthermore,
\[
    \widehat{d}(\varphi_H^1,\varphi_{\widetilde{H}}^1)  \leq
    \widehat{d}(\varphi_H^1,\varphi_{H'}^1) +\widehat{d}(\varphi_{H'}^1,\varphi_{\widetilde{H}}^1) \leq \varepsilon_1 +  \frac{1}{k^2}
\]
and
\SSplit{
    \widehat{d}\left(\varphi_{\Sigma(\widetilde{H})}^1,\varphi_{\Sigma(H)}^1\right)
    &\leq
    \widehat{d}\left(\varphi_{\Sigma(\widetilde{H})}^1,\varphi_{\Sigma(H')}^1\right)+
    \widehat{d}\left(\varphi_{\Sigma(H')}^1,\varphi_{\Sigma(H)}^1\right)\\ 
    &\leq\widehat{d}\left(\varphi_{\Sigma(H')}^1,\varphi_{\Sigma(H)}^1\right) + \frac{1}{k}.
}
By the \emphref{prop:Holder} property  of Theorem \ref{thm: prop},
\[
    \widehat{d}\left(\varphi_{\Sigma(H')}^1,\varphi_{\Sigma(H)}^1\right) \leq C'\varepsilon_1^{\frac{1}{2}}\left(\varepsilon_1 +2\norm{H}_{C_2}\right)^{\frac{1}{2}}.
\]
In total, we obtain
\[
    \widehat{d}(\varphi^1_{{H}},\psi^{-1}\varphi^1_{\Sigma({H})}\psi) \leq 18 \Area(\S^2) + \varepsilon_1 +  \frac{1}{k^2}  +\frac{1}{k} + C'\varepsilon_1^{\frac{1}{2}}\left(\varepsilon_1 +2\norm{H}_{C_2}\right)^{\frac{1}{2}}.
\]
Note that $\psi$ depends on $\widetilde{H}$, ultimately depending on $\varepsilon_1$ and $k$. Taking $\varepsilon_1$ arbitrarily small and $k$ arbitrarily large we get that
\[
    \forall \varepsilon>0, \; \exists \psi \in \widehat{\Ham}(\S^2),\qquad 
    \widehat{d}(\varphi^1_{{H}},\psi^{-1}\varphi^1_{\Sigma({H})}\psi) \leq 18 \Area(\S^2) +\varepsilon.
\]
Taking $\varepsilon<\Area(\S^2)$ we obtain the \emphref{prop:ConjugationControl} property of  Theorem \ref{thm: prop}, but also, if $\Sigma(H)\equiv 0$, we have 
\[
    \widehat{d}\left(\varphi^1_{{H}},\underbrace{\psi^{-1}\varphi^1_{\Sigma({H})}\psi}_{\id}\right) \leq 19 \Area(\S^2),
\]
and since $\Sigma(tH)=t\Sigma(H)$, this proves what remained to show of the \emphref{prop:BoundedGrowthControl} property. \end{proof}


\subsection{The proof of the conjugation lemma}\label{subsec: thm fqm}


The following more flexible version of Lemma \ref{lma: Sik} will be needed in the proof of the conjugation lemma:

\begin{lemma}\label{lma: Sikflex}
    Let $(M,\omega)$ be a closed and connected symplectic surface. Let $\varepsilon>0,$ let $m$ be a positive integer, and let 
    $\mathcal{D}_0,...,\mathcal{D}_m \subset M$ be disjoint topological open disks, such that $ \mathcal{D}_0 $ is of area $\varepsilon$, and 
    for each $ 1 \leqslant i \leqslant m $, the area of $ \mathcal{D}_i $ lies in $ (\varepsilon/2,\varepsilon] $. Finally, let $\psi_j: \mathcal{D}_j \rightarrow \mathcal{D}_0$ for all $1\leq j \leq m$ be symplectic embeddings. Consider $f_0,...,f_m\in \Ham(M,\omega)$ with ${\supp(f_j) \subset \mathcal{D}_j.}$ Define $\Phi,\Phi'\in \Ham(M,\omega)$ by
    \[
    \Phi = f_0f_1\cdots f_m
    \]
    and
    \[
        \Phi' = f_0 \prod_{j=1}^m (\psi_j)_* f_j
    \]
    where $(\psi_j)_*f_j$ is given by $(\psi_j)_* f_j = \psi_j f_j (\psi_j)^{-1}$ on $\psi_j(\mathcal{D}_j) \subset \mathcal{D}_0$ and $(\psi_j)_*f_j = \id$ on $M\setminus \psi_j(\mathcal{D}_j)$. Then
    \[
    d(\Phi,\Phi') < 7\varepsilon.
    \]
\end{lemma}
\begin{proof}
Without loss of generality we may assume that each $ \mathcal{D}_j $ has a smooth boundary and its area lies in $ (\varepsilon/2,\varepsilon) $, and moreover each $ \psi_j $ extends to a smooth symplectic embedding $ \psi_j: \overline{\mathcal{D}_j} \rightarrow \mathcal{D}_0$.

First consider the case when $ m = 2k $ is even. Look at the disks $ \mathcal{D}_0, \ldots, \mathcal{D}_k $. The area of the complement to the union of these disks is greater than or equal to the sum of areas of $ \mathcal{D}_{k+1}, \ldots, \mathcal{D}_m $, which in turn is greater than $ k\varepsilon/2 $. 
Moreover, recall that the area of each disk $ \mathcal{D}_j $ lies in $ (\varepsilon/2,\varepsilon] $. Hence for each $ 1 \leqslant j \leqslant k $ we can find 
an open disk $ \mathcal{D}_j' \supset \overline{\mathcal{D}}_j$ of area $ \varepsilon $, such that each $ \psi_j $ can be extended to a symplectomorphism 
$ \psi_j : \mathcal{D}_j' \rightarrow \mathcal{D}_0 $, and such that moreover the disks $ \mathcal{D}_0, \mathcal{D}_1', \ldots, \mathcal{D}_k' $ are 
pairwise disjoint. Applying Lemma \ref{lma: Sik}, we conclude that for $ \Phi_0 := f_0 \cdots f_k $ and $ \Phi_0' = f_0 \prod_{j=1}^k (\psi_j)_* f_j $ we have
$ d(\Phi_0,\Phi_0') < 3\varepsilon $. Now, looking at the disks $ \mathcal{D}_0, \mathcal{D}_{k+1}, \ldots, \mathcal{D}_m $, and applying a similar reasoning, for $ \Phi_1 := f_{k+1} \cdots f_m $ and $ \Phi_1' = \prod_{j=k+1}^m (\psi_j)_* f_j $ we get $ d(\Phi_1,\Phi_1') < 3\varepsilon $. Finally, we conclude $ d(\Phi,\Phi') = d(\Phi_0\Phi_1,\Phi_0'\Phi_1') \leqslant d(\Phi_0,\Phi_0') + d(\Phi_1,\Phi_1') < 6\varepsilon < 7\varepsilon $.

It remains to verify the case when $ m = 2k+1 $ is odd. Denote $ \Psi := f_0 f_1\cdots f_{m-1} $ and $ \Psi' = f_0 \prod_{j=1}^{m-1} (\psi_j)_* f_j $. We have shown that $ d(\Psi,\Psi') < 6\varepsilon $. Since the area of $ \mathcal{D}_m $ is less than $ \varepsilon $ and $ \psi_m(\overline{\mathcal{D}}_m) \cap \overline{\mathcal{D}}_m = \emptyset $, we can find a Hamiltonian diffeomorphism $ \phi \in \Ham(M,\omega) $ with $ d(\id,\phi) < \varepsilon/2 $, such that $ \phi = \psi_m $ on $ \mathcal{D}_m $, and in particular $ (\psi_m)_* f_m = \phi^{-1} f_m \phi $, which implies $ d(f_m,  (\psi_m)_* f_m) = d(f_m,\phi^{-1} f_m \phi ) \leqslant 2 d(\id, \phi) < \varepsilon $. We conclude $ d(\Phi,\Phi') = d(\Psi f_m, \Psi' (\psi_m)_* f_m) <  d(\Psi, \Psi') +  d(f_m,  (\psi_m)_* f_m) < 7 \varepsilon $.

\end{proof}

\begin{wrapfigure}{r}{0.4\textwidth}
    \tdplotsetmaincoords{80}{110}
        \begin{tikzpicture}[scale=2,tdplot_main_coords,sphere segment/.style args={%
            phi from #1 to #2 and theta from #3 to #4 and radius #5}{insert path={%
             plot[variable=\x,smooth,domain=#2:#1] 
             (xyz spherical cs:radius=#5,longitude=\x,latitude=#3)
             -- plot[variable=\x,smooth,domain=#3:#4] 
             (xyz spherical cs:radius=#5,longitude=#1,latitude=\x)
             --plot[variable=\x,smooth,domain=#1:#2] 
             (xyz spherical cs:radius=#5,longitude=\x,latitude=#4)
             -- plot[variable=\x,smooth,domain=#4:#3] 
             (xyz spherical cs:radius=#5,longitude=#2,latitude=\x)}}]
            \tdplotsetrotatedcoords{20}{80}{0}
                \draw [ball color=white,very thin,tdplot_rotated_coords] (0,0,0) circle (1) ;
                \draw[thin,fill=red!40!white,fill opacity=0.6,
                sphere segment={phi from -20 to 160 and theta from 17.46 to 44.5  and radius 1}] ;

                \draw [domain=-75:115] plot ({0.953*cos(\x)}, {0.953*sin(\x)},0.3);
                \draw [dashed,domain=115:285] plot ({0.953*cos(\x)}, {0.953*sin(\x)},0.3);
                \draw [domain=-75:115] plot ({0.714*cos(\x)}, {0.714*sin(\x)},0.7);
                \draw [dashed,domain=115:285] plot ({0.714*cos(\x)}, {0.714*sin(\x)},0.7);

                \node at (0,-1.2,0.2) {$C_+$};
                \node at (0,-1.05,0.6) {$C_-$};

                \node at (0,1.4,0.5) {$A_i$};
                \draw [color=black] (0,1.2,0.3) -- (0,1.2,0.7);
                \draw [color=black] (0,1.2,0.3) -- (0,1.1,0.3);
                \draw [color=black] (0,1.2,0.7) -- (0,1.1,0.7);

                \node at (0,1.6,0.85) {$D_-(A)$};
                \draw [color=black] (0,1.2,0.7) -- (0,1.2,1);
                \draw [color=black] (0,1.2,0.7) -- (0,1.1,0.7);
                \draw [color=black] (0,1.2,1) -- (0,1.1,1);

                \node at (0,2.4,0.7) {$D_+(A)$};
                \draw [color=black] (0,2,0.3) -- (0,2,1);
                \draw [color=black] (0,2,0.3) -- (0,1.9,0.3);
                \draw [color=black] (0,2,1) -- (0,1.9,1);

        \end{tikzpicture}
        
\end{wrapfigure}
\textit{Proof of \hyperref[ConjugationLemma]{the conjugation lemma.}}
Decompose $\widetilde{H} = \sum_A \widetilde{H}_A$, like before, where $\left\{A\right\}$ are annuli/disks corresponding to the edges of the Reeb graph of $H'$, and $\widetilde{H}_{A}$ is locally constant outside $A.$ Crucially, since we assumed $r$ is flat near the median level, each $A$ lies in a disk of area $< \frac{1}{2}\Area(\S^2).$ Let $\partial A$ be the union of two disjoint circles, $ C_+, C_-.$ We denote by $D_+(A)$ the disk of area $< \frac{1}{2} \Area({\S}^{2})$ containing $A,$ whose boundary is $C_+.$ Similarly, we denote by $D_-(A)$ the disk of area $< \frac{1}{2} \Area({\S}^{2})$ bounded by $C_-.$ If $A$ is a disk, $D_-(A) =\emptyset$. Let $m$ be the median level of $H'$, and let $\varepsilon>0$ such that $r|_{[m-2\varepsilon,m+2\varepsilon]} = \const$. The connected components of ${\S}^{2} \setminus (H')^{-1}(m-\varepsilon,m+\varepsilon)$ are displaceable, hence one can find a collection of disjoint disks $D_1,...,D_M$ with smooth boundary, each of Area $< \frac{1}{2}\Area({\S}^{2})$ such that 
\[
    {\S}^{2}\setminus(H')^{-1}(m-\varepsilon,m+\varepsilon) \subset \bigcup_j D_j.
\]
Order them such that $\Area(D_1)\geq \Area(D_j)$ for all $j$, and by conjugating with some $\psi \in \Ham({\S}^{2})$, we can assume $D_1$ is the standard cap
\[
D_1 = \{(x,y,z) \in \S^{2} | z> z_1\} ,\qquad z_1 \in \left(0,\frac{1}{2}\right).
\]
Note that the time-1 map is $\varphi_{\widetilde{H}}^1 = \prod_i   \varphi_{\widetilde{H}_i}^1$.
We will now deal the annuli which lie inside and outside $D_1$ separately. Before starting let us define the following.
\begin{defin}
    A collection of annuli $C = \left(A_1,...,A_l\right)$ is called \underline{ordered} if \[D_-(A_i) \supset D_{+}(A_{i+1}), \quad \forall 1 \leq i < l.\]
\end{defin}
 \begin{center}
     
\begin{tikzpicture}[line cap=round,line join=round,>=triangle 45,x=1.0cm,y=1.0cm]

    \draw [name path=B, thick, draw,fill=blue!20](0,0) circle (2cm);    
    \draw [name path=A, thick, draw,fill=white](0,0) circle (1.5cm);   
    \draw [name path=B, thick, draw,fill=blue!20](0,0) circle (1cm);    
    \draw [name path=A, thick, draw,fill=white](0,0) circle (0.5cm);
    
    \draw [name path=B, thick, draw,fill=red!20](6,0) circle (2cm);    
    \draw [name path=A, thick, draw,fill=white](6,0) circle (1.8cm);   
    \draw [name path=B, thick, draw,fill=red!20](5,0) circle (0.5cm);    
    \draw [name path=A, thick, draw,fill=white](5,0) circle (0.3cm);
    \draw [name path=B, thick, draw,fill=red!20](6,1) circle (0.5cm);    
    \draw [name path=A, thick, draw,fill=white](6,1) circle (0.3cm);
    \draw [name path=B, thick, draw,fill=red!20](7,-0.5) circle (0.5cm);    
    \draw [name path=A, thick, draw,fill=white](7,-0.5) circle (0.3cm);
    \draw (0,-2.5) node {Ordered};    
    \draw (6,-2.5) node {Not Ordered};    
    \end{tikzpicture}
 \end{center}

 \begin{rem}
 Note that an ordered collection of annuli is a chain in the partially ordered set of annuli contained in a disk of area $< \frac{1}{2} \Area({\S}^{2}),$ with partial order $A_1 < A_2$ iff $D_-(A_1) \supset D_+(A_2).$
 \end{rem}
\underline{Let us start with the annuli lying on the complement of $D_1$:}\\
Chose $\varepsilon = \frac{1}{2^N}\Area({\S}^{2})$ such that $\varepsilon < \Area(D_+(A))$ for all annuli $A\subset \S^2\setminus D_1$.\\
We will iterate the following step $N-1$ times:\\
For each $j =0,...,N-2$, consider all annuli $A\subset {\S}^{2}\setminus D_1$ such that $\Area(D_+(A)) \in (2^j\varepsilon,2^{j+1}\varepsilon].$ Denote the collection of these annuli by $C_j.$ From the constraint on the area of $D_+(A)$, we can partition $C_j$ to be $C_j = C_{j1} \cup ... \cup C_{jk_j}$, such that each $C_{jk}$ is an ordered collection of annuli. If $A \in C_{jk_1}$ and $B \in C_{jk_2}$ for $k_1 \neq k_2,$ then the disks $D_+(A)$ and $D_+(B)$ are disjoint and are also disjoint from $D_1$. Now we may use  Lemma \ref{lma: Sikflex} in the following way. Denote $\widetilde{H}_{jk} = \sum_{A\in C_{jk}} \widetilde{H}_A$, and for each $1 \leq k \leq k_j$, find $\varphi_{k,j} \in \Ham(\S^{2})$ such that $\varphi_{k,j}^*\widetilde{H}_{jk} $ is a function of $z$, constant outside $D_1.$ Then if we denote
\[
    \widetilde{H}^j := \sum_{k=1}^{k_j} \widetilde{H}_{jk} ,\qquad \widehat{H}^j :=  \sum_{k=1}^{k_j} \varphi_{k,j}^*\widetilde{H}_{jk},
\]
then we have
\[
    d\left(\varphi^1_{\widetilde{H}^j},\varphi^1_{\widehat{H}^j}\right) < 7\cdot 2^{j+1}\varepsilon < 2^{j+4}\varepsilon.
\]
Hence in total, after $N-1$ steps, the functions
\[
    \widetilde{H}_{\text{out}} := \sum_{j=0}^{N-2} \widetilde{H}^j = \sum_{A \subset S^2 \setminus D_1} \widetilde{H}_A ,\qquad
    \widehat{H}_{\text{out}} := \sum_{j=0}^{N-2} \widehat{H}^j,
\]
satisfy 
\begin{equation} \label{eq:H-out}
    d(\varphi^1_{\widetilde{H}_{\text{out}} },\varphi^1_{ \widehat{H}_{\text{out}} }) \leq \sum_{j=0}^{N-2}2^{j+4}\varepsilon < 2^{N+3}\varepsilon = 8 \Area(\S^{2}).
\end{equation}


\underline{For the annuli lying in $D_1$:}\\
Denote $ \widetilde{H}_{\text{in}} := \sum_{A \subset D_1} \widetilde{H}_A $, and by applying the transformation\footnote{Recall that $ R $ is given by $(x,y,z) \mapsto (-x,y,-z)$, via the embedding $\S^{2} \subset \R^3$ as a sphere of radius $1/2$.} $ R \in \Ham(S^2) $, all the relevant annuli are moved outside of $ D_1 $, which enables us to use the considerations of the previous case. More concretely, consider
$$ R^*  \widetilde{H}_{\text{in}} := \sum_{A \subset D_1} R^* \widetilde{H}_A ,$$
where for each annulus $ A \subset D_1 $, the function $ R^* \widetilde{H}_A $ locally constant outside the annulus $ R(A) $ that lies in $ S^2 \setminus D_1 $. Now by applying the arguments from the previous case to the function $ R^*  \widetilde{H}_{\text{in}} $, we obtain a smooth function $ \widehat{H}_{\text{in}} =  \widehat{H}_{\text{in}} (z) $ which is constant outside $ D_1 $, such that 
$$ d(\varphi^1_{R^*\widetilde{H}_{\text{in}} },\varphi^1_{ \widehat{H}_{\text{in}} }) < 8 \Area(\S^{2}). $$
Since in addition we have $$ d(\varphi^1_{R^*\widetilde{H}_{\text{in}} },\varphi^1_{\widetilde{H}_{\text{in}} }) = d(R \varphi^1_{\widetilde{H}_{\text{in}} } R,\varphi^1_{\widetilde{H}_{\text{in}} }) \leq 2 d(R,\id) =  \Area(\S^{2}), $$
we get 
\begin{equation} \label{eq:H-in}
d(\varphi^1_{\widetilde{H}_{\text{in}} },\varphi^1_{ \widehat{H}_{\text{in}} }) < 9 \Area(\S^{2}).
\end{equation}

We have $ \widetilde{H} =   \widetilde{H}_{\text{out}}+\widetilde{H}_{\text{in}} $, and denoting $ \widehat{H} =     \widehat{H}_{\text{out}}+\widehat{H}_{\text{in}} $, by $ (\ref{eq:H-out}) $ and $ (\ref{eq:H-in}) $ we get 
$$ d(\varphi^1_{\widetilde{H}},\varphi^1_{\widehat{H}}) < 17 \Area(\S^{2}). $$
Moreover, we have $ \Sigma (\widetilde{H}) = \frac{1}{2} \widehat{H} + \frac{1}{2} R^*\widehat{H} $ and consequently
$$ d\left(\varphi_{\widehat{H}}^1, \varphi_{\Sigma(\widetilde{H})}^1 \right) = d\left(\varphi_{\widehat{H}}^1, \varphi_{\widehat{H}}^{\frac12} R \varphi_{\widehat{H}}^{\frac12} R \right) \leq 2 d(R,\id) = \Area(\S^{2}). $$
We conclude 
$$ d\left(\varphi_{\widetilde{H}}^1, \varphi_{\Sigma(\widetilde{H})}^1 \right) \leq d(\varphi^1_{\widetilde{H}},\varphi^1_{\widehat{H}}) + d\left(\varphi_{\widehat{H}}^1, \varphi_{\Sigma(\widetilde{H})}^1 \right) < 18 \Area(\S^{2}), $$
thus completing the proof.
%
%
\qed


\section{Discussion}

\subsection{Surfaces other than $S^2$}\label{sec: other surfaces}
The case of the two-disk and compactly supported Hamiltonian diffeomorphisms is very similar to that of the sphere. Furthermore, we expect our methods to prove a growth dichotomy statement for all compact surfaces (with or without boundary). We plan to attack this question by means of symmetrization based on partial quasi-morphisms from \cite{MT}, which should apply to Hamiltonians with all regular components of level sets contractible. Hamiltonians $H$ without this property are easily seen to have linear Hofer growth $\r(H)>0$ for instance by the energy capacity inequality in the universal cover (cf. \cite[Exercise 7.2.E]{P-book}).

To make this more concrete, take a closed symplectic surface $(M,\omega)$ of total area $1$. We shall use the following ``Cartan subalgebra" in $C^0(M)$. Fix a symplectic embedding of the open Euclidean disk of area $1,$ $$\imath: (D^2, dp \wedge dq) \to (M,\omega),$$
where $M \setminus \imath(D^2)$ is a collection of closed simple curves on $M$. Think of the interior of the standard fundamental domain of $M$ on the universal cover. Write $z = (p^2+q^2)/2$ for the symplectic polar radius. Denote by $\mathcal{F} \subset C^0_0(M)$ the set of continuous mean zero functions on $M$ of the form $\imath_*F - \int_{D^2} F \omega$, where $F$ runs over the set of smooth compactly supported functions on $D^2$ depending only on $z$. 

Denote by $\mathcal{C}$ the set of mean zero autonomous Hamiltonians on $M$
whose  regular level sets are contractible. Define
the symmetrization map $\Sigma: \mathcal{C} \to  \mathcal{F}$. For suitably flattened functions (as in Lemma \ref{lem:flattening-conv}, see also Remark \ref{rem: flattening other} below) it is defined analogously to the case of $S^2$, and in the general case one uses the local quasi-morphisms constructed in \cite{MT} similarly to what we did in Section \ref{subsec: symm qm} above. 

The conclusion is that for $H \in \mathcal{C}$ the growth is linear if $\Sigma(H) \neq 0$,
and it is bounded if $\Sigma(H)=0$. As explained above, this proves a growth dichotomy statement for arbitrary autonomous Hamiltonians on $M.$ 
We expect that an enhanced dichotomy also holds in this setting. The details will appear in a forthcoming paper. 

\begin{rem}\label{rem: flattening other}
Crucially, the flattening of Lemma \ref{lem:flattening-conv} would allow us to proceed with the proof of the soft part of the dichotomy on $ M $, since given $ H $ with contractible regular level sets, we can directly pass to $ r \circ H $ without the need to approximate $ H $ by a Morse function with contractible regular level sets. 
\end{rem}


\subsection{Quasi-morphisms and autonomous Hamiltonian diffeomorphisms}
We note that, as discussed in Remark \ref{rem:deltas}, given measures $\delta_i = \delta_{k_i,B_i}$ on $(-1/2,1/2)$ such that $\sum a_i \delta_i = 0$ for certain $a_i \in \R,$ 
the quasimorphism \begin{equation}\label{eq mu}\mu = \sum a_i \mu_{k_i,B_i}: \Ham(S^2) \to \R\end{equation} vanishes on the subgroup $\cl T$ of $\Ham(S^2)$ given by autonomous Hamiltonians generated by functions of the height $z$ only. As $\mu$ is Lipschitz in the Hofer metric $d,$ it extends to the Hofer completion $\wh{\Ham}(S^2)$ and vanishes on $\wh{\cl T}.$ Therefore, it follows from Theorem \ref{thm-2} that $\mu$ vanishes on the subset $\Aut(S^2) \subset \Ham(S^2)$ of all autonomous Hamiltonian diffeomorphisms (those generated by time-independent Hamiltonians $H \in \sm{S^2,\R}$). It is also follows immediately from \cite[Theorem 1.7]{EPP} that $\mu$ is continuous with respect to the $C^0$ topology on $\Ham(S^2)$ and extends to the $C^0$-closure $\ol{\Ham}(S^2) $ of $\Ham(S^2)$ inside $\mrm{Homeo}(S^2).$ A particular example of such a quasimorphism $\mu$ is given by \[\mu = 3\mu_{3,B}-2\mu_{2,B}-\mu_{1,1/2}\] for $1/3<B<1/2,$ where $\mu_{1,1/2}$ is the Entov-Polterovich quasimorphism \cite{EP-qm}.

In future work joint with P. Haim-Kislev we plan to attack the following question.

\begin{question}\label{question mu}
Do all the quasimorphisms $\mu: \Ham(S^2) \to \R$ from \eqref{eq mu} vanish identically?
\end{question}

We note that if one such $\mu$ does not vanish, then the Hofer distance \[\mrm{aut}(S^2) = \sup_{\phi \in \Ham(S^2)} d(\phi, \Aut(S^2))\] is infinite, which would settle the last open case  of a conjecture of Polterovich-Shelukhin \cite{PSh-pmod} for surfaces. Moreover, for all natural $k,$ \[\mrm{aut}_k(S^2) = \sup_{\phi \in \Ham(S^2)} d(\phi, \Aut(S^2)^k) = \infty,\] similar to the case of $L^p$-metrics as in \cite{BS-aut}.

If on the other hand all such $\mu$ vanish, then the symmetrization map \[\Aut(S^2) \to \cl T_{ev}\] extends to all $\Ham(S^2),$ and it would be very interesting to study its properties.

\bibliographystyle{plain}
\bibliography{Symmrefs}

\end{document}